\documentclass[11pt, reqno]{amsart}
\usepackage{amsfonts,latexsym,enumerate}
\usepackage{amsmath}
\usepackage{amscd}
\usepackage{float,amsmath,amssymb,mathrsfs,bm,multirow,graphics}
\usepackage[dvips]{graphicx}
\usepackage[percent]{overpic}
\usepackage[pdftex]{color}
\usepackage{amsaddr}
\usepackage[numbers,sort&compress]{natbib}
\usepackage{todonotes}
\usepackage{tikz}
\usetikzlibrary{decorations.markings}

\newcommand{\nequation}{\setcounter{equation}{0}}
\renewcommand{\theequation}{\mbox{\arabic{section}.\arabic{equation}}}
\newcommand{\R}{{\Bbb R}}

\newcommand{\C}{{\Bbb C}}

\newcommand{\proofend}{\hfill$\Box$\bigskip}
\newcommand{\proofendcontinue}{\hfill \raisebox{.8mm}[0cm][0cm]{$\bigtriangledown$}\bigskip}

\newcommand{\tr}{\text{\upshape tr\,}}
\newcommand{\res}{\text{\upshape Res\,}}

\newcommand{\re}{\text{\upshape Re\,}}

\newcommand{\im}{\text{\upshape Im\,}}

\newcommand{\E}{\mathcal{E}}
\renewcommand{\Re}{\text{Re }}
\newcommand{\U}{\mathsf{U}}
\newcommand{\V}{\mathsf{V}}
\newcommand{\Surface}{\mathcal{S}}

\DeclareMathOperator{\Int}{int}

\def\XXint#1#2#3{{\setbox0=\hbox{$#1{#2#3}{\int}$}
\vcenter{\hbox{$#2#3$}}\kern-.5\wd0}}

\newtheorem{theorem}{Theorem}

\newtheorem{proposition}{Proposition}[section]

\newtheorem{lemma}[proposition]{Lemma}

\numberwithin{equation}{section}

\input epsf
\date{\today}
\title[The hyperbolic Ernst--Maxwell equations in a triangular domain]
{The hyperbolic Ernst--Maxwell equations \\ in a triangular domain}

\author{Julian Mauersberger}
\address{Department of Mathematics, KTH Royal Institute of Technology, \\ 100 44 Stockholm, Sweden.}
\email{julianma@kth.se}

\begin{document}
\begin{abstract} 
\noindent In a recent paper, we applied Riemann--Hilbert techniques to analyze the Goursat problem for the hyperbolic Ernst equation, which describes the interaction of two colliding gravitational plane waves. Here we generalize this approach to colliding electromagnetic plane waves in Einstein--Maxwell theory. 
\end{abstract}

\maketitle



\setcounter{tocdepth}{1}

\section{Introduction}
In 1968, F. J. Ernst found \cite{E1968b} that each stationary axisymmetric solution of the full vacuum Einstein--Maxwell equations can be described by a system of two nonlinear equations. S. Chandrasekhar and B. C. Xanthopoulos \cite{CX1985} showed later that a similar reduction can be established in the case of colliding electromagnetic plane waves resulting in an integrable hyperbolic system of two nonlinear equations that can be written as
\begin{align}\label{ErnstMaxwellEquations}
\begin{cases}
(\Re \E - H\bar{H}) \left( \E_{xy} - \frac{\E_x+\E_y}{2(1-x-y)} \right) &= \E_x \E_y - \bar{H} (\E_x H_y+\E_yH_x), \\ 
(\Re \E - H\bar{H})\left( H_{xy} - \frac{H_x+H_y}{2(1-x-y)} \right)&= \frac{1}{2} (\E_x H_y + \E_y H_x)-2\bar{H} H_xH_y, 
\end{cases}
\end{align}
where the Ernst potentials $\mathcal{E}(x,y)$ and $H(x,y)$ are complex-valued functions of two real variables. In particular, it turns out that the interaction of two gravitational plane waves, each of which supports an electromagnetic shock wave, reduces to a Goursat problem for \eqref{ErnstMaxwellEquations} in the triangular region (see Figure \ref{figD})
\begin{align}\label{Ddef}
D=\{ (x,y) \in \R^2 \, | \, x\ge0, y\ge0, x+y<1 \}.
\end{align}
\begin{figure}
	\begin{center}
		\begin{tikzpicture}	
			\draw[white,shade, inner color = gray!25, outer color = gray!25] (-1.5,0) -- (-1.5,4.5) -- (3,0) -- (-1.5,0) ;
		\draw[very thick,->] (-2,0) -- (4,0);
		\draw[very thick,->] (-1.5,-0.5) -- (-1.5,5.5);
		\node at (4,-0.3) {$x$};
		\node at (-1.8,5.5) {$y$};
		
		\draw[dashed] (-1.5,4.5) -- (3,0);
		\node at (-1.8,4.5) {$1$};
		\node at (3,-0.3) {$1$};
		\node at (0,1.5) {$D$};	

		\node[rotate=90] at (-2.3,2) {\small $\E(0,y)=\E_1(y)$};
		\node[rotate=90] at (-1.8,2) {\small $H(0,y)=H_1(y)$};
		\node at (1,-0.3) {\small $\E(x,0)=\E_0(x)$};
		\node at (1,-0.8) {\small $H(x,0)=H_0(x)$};
		
		\node at (3,-0.3) {$1$};		
		\end{tikzpicture}
\caption{The triangular region $D$ defined in \eqref{Ddef} and the Goursat problem for the Ernst--Maxwell equations \eqref{ErnstMaxwellEquations}.} \label{figD}
\end{center}
\end{figure}
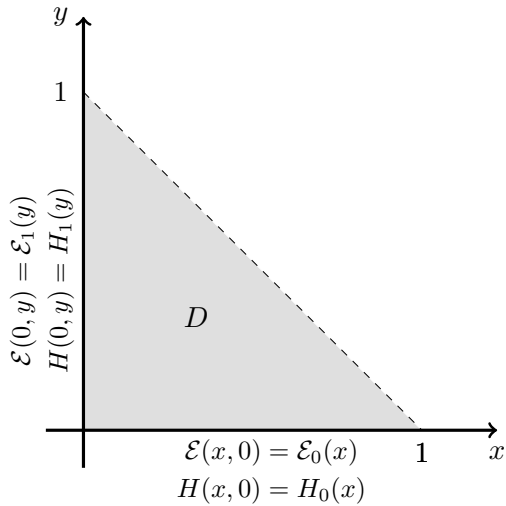
We seek (sufficiently regular) functions $\mathcal{E}(x,y)$ and $H(x,y)$ such that
\begin{align*}
\begin{cases}
\text{$\E$ and $H$ satisfy the equations \eqref{ErnstMaxwellEquations} in $D$,} \\
\mathcal{E}(x,0)=\mathcal{E}_0(x), \, H(x,0)=H_0(x), \quad x \in [0,1), \\ \mathcal{E}(0,y)=\mathcal{E}_1(y), \hspace{0.13cm} H(0,y)=H_1(y), \hspace{0.05cm} \quad y \in [0,1),
\end{cases}
\end{align*}
for some given boundary data $\{ \E_0, \E_1,H_0,H_1 \}$. 

In case of $H=0$, the equations \eqref{ErnstMaxwellEquations} reduce to the single hyperbolic Ernst equation. This equation is a variation of the original elliptic Ernst equation \cite{E1968}. The elliptic version of Ernst's equation describes stationary axisymmetric spacetimes whereas the hyperbolic version describes the interaction of colliding gravitational plane waves. The Goursat problem for the hyperbolic Ernst equation was analyzed in the recent paper \cite{LMernst} by using the integrable structure of the equation. 

In this paper, we generalize the approach adopted in \cite{LMernst} to the case of colliding electromagnetic plane waves. The Ernst--Maxwell equations \eqref{ErnstMaxwellEquations} admit a Lax pair involving $3 \times 3$-matrices, whereas the Lax pair for Ernst's equation is given in terms of $2 \times 2$-matrices. Nevertheless, a similar framework can be established here. However, due to the more complicated structure of the equations and the Lax pair, many steps become more involved.

The existing literature on the equations \eqref{ErnstMaxwellEquations} (see e.g. \cite{BS1974, CX1985, H1990}) focuses on the generation of new exact solutions instead of a general approach for solving the Goursat problem for \eqref{ErnstMaxwellEquations} as it is established here. However, besides \cite{LMernst}, inverse scattering approaches have been used in order to solve a boundary problem for colliding gravitational waves \cite{FST1999} and more recently to analyze boundary value problems for the elliptic version of Ernst's equation \cite{LF2011,L2011,LP2018}. All the papers \cite{FST1999,LF2011,L2011,LMernst,LP2018} are partially based on a general framework for solving boundary value problems known as the unified transform or Fokas method \cite{F1997} (see also \cite{BFS2006, DTV2014}). The unified transform has been applied to a variety of different integrable systems (see e.g.~\cite{BdMS2003,FIS2005,PP2010}).

In order to be relevant in Einstein--Maxwell theory, a solution of the Goursat problem for \eqref{ErnstMaxwellEquations} must satisfy the boundary conditions (see \cite{Griffiths1991} and appendix)
\begin{subequations}\label{boundaryvalues}
\begin{align} 
&(1-y)	\lim_{x\to 0} x \left( \frac{|\mathcal{E}_x(x,y)-2\bar{H}_1(y)H_x(x,y)|^2}{f_1(y)^2} + \frac{4 |H_x(x,y)|^2}{f_1(y)} \right)=2k_1,
\\
&(1-x)	\lim_{y\to 0} y  \left( \frac{|\mathcal{E}_y(x,y)-2\bar{H}_0(x)H_y(x,y)|^2}{f_0(x)^2} + \frac{4 |H_y(x,y)|^2}{f_0(x)} \right)=2k_2,
\end{align}
\end{subequations}
for some constants $k_1, k_2 \in [1/2,1)$. It is therefore crucial to allow for data with derivatives that are singular at the corners of the triangular domain $D$. Moreover, for a relevant solution $\{ \E, H \}$ of \eqref{ErnstMaxwellEquations}, the term $\Re \mathcal{E} -|H|^2$ is strictly positive. If $ \mathcal{E}$ and $H$ solve the equations \eqref{ErnstMaxwellEquations}, then other solutions of \eqref{ErnstMaxwellEquations} are given by (cf. \cite{Griffiths1991})
\begin{align*}
\{ \E+2\bar{\alpha}H+|\alpha|^2,H+\alpha \}, \{ |\beta|^2 \E, \beta H \}, \{ \mathcal{E}+ic,H \},
\end{align*}
where $\alpha, \beta$ are complex constants and $c$ is a real constant. Hence we may assume that $H(0,0)=0$ and $\mathcal{E}(0,0)=1$. This motivates our assumptions on the boundary data:
\begin{align}\label{AssumptionsBoundary}
\begin{cases}
\E_0,H_0, \E_1, H_1 \in C([0,1)) \cap C^n((0,1)), 
\\
\text{$x^\alpha \E_{0x},x^\alpha H_{0x},  y^\alpha \E_{1y},y^\alpha H_{1y} \in C([0,1))$,} 
\\
\E_0(0) = \E_1(0) = 1,
\\
H_0(0)=H_1(0)=0,
\\
\text{$\Re \E_0(x) -|H_0(x)|^2 >0$ for $x\in [0,1)$, }
\\
\text{$\Re \E_1(y) -|H_1(y)|^2 >0$ for $y\in [0,1)$, }
\end{cases}
\end{align}
for some $\alpha \in [0,1)$ and some integer $n\ge 2$.

In Section \ref{laxpairRHsection}, we introduce the tools that are necessary in the analysis of the Goursat problem for \eqref{ErnstMaxwellEquations} and their properties. This includes the Lax pair for \eqref{ErnstMaxwellEquations} and its spectral data as well as a corresponding Riemann--Hilbert (RH) problem. 

In Section \ref{mainresultssec}, we present the three main results of the paper. Theorem \ref{mainth1} treats uniqueness of the solution of the Goursat problem for \eqref{ErnstMaxwellEquations} and its representation in terms of a RH problem, whose formulation only involves the given boundary data. Theorem~\ref{mainth2} is an existence and regularity result for the same Goursat problem and in Theorem \ref{mainth3} we consider the boundary conditions \eqref{boundaryvalues}. 

Section \ref{proofssec} is devoted to the proofs of the main results.

In Section \ref{examplessec}, we present a simple family of examples for solutions of the Goursat problem for \eqref{ErnstMaxwellEquations} and compute their boundary values \eqref{boundaryvalues}. In the appendix, we explain how \eqref{ErnstMaxwellEquations} and \eqref{boundaryvalues} can be derived from the formulas given in the literature.

\section{Prerequisites and notation}\label{laxpairRHsection}
In this section, we present the Lax pair and a Riemann--Hilbert problem corresponding to the hyperbolic Ernst--Maxwell equations \eqref{ErnstMaxwellEquations}. In this context, we also introduce the notation used throughout the paper, which is adopted from \cite{LMernst}.
\subsection{Lax pair and spectral data}
The equations \eqref{ErnstMaxwellEquations} admit a Lax pair given by
\begin{align}\label{LaxPair}
\begin{cases}
\Phi_x(x,y,k) =  \U(x,y,k)\Phi(x,y,k),
\\
\Phi_y(x,y,k) =  \V(x,y,k)\Phi(x,y,k),
\end{cases}
\end{align}
where $k$ is a spectral parameter, $\Phi$ is a $3 \times 3$-matrix valued eigenfunction, and the $3 \times 3$-matrix valued functions $\U(x,y,k)$ and $\V(x,y,k)$ are defined by
\begin{align*}
\U= \begin{pmatrix}
\bar{A} & \lambda \bar{A} & \frac{i}{\sqrt{f}} H_x
\\
\lambda  A & A & -\lambda \frac{i}{\sqrt{f}} H_x
\\
\frac{i}{\sqrt{f}} \bar H_x & \lambda  \frac{i}{\sqrt{f}} \bar H_x & \frac{1}{2} (A + \bar A)
\end{pmatrix}, \qquad
\V= \begin{pmatrix}
\bar B & \frac{1}{\lambda} \bar B & \frac{i}{\sqrt{f}} H_y
\\
\frac{1}{\lambda}  B & B & - \frac{i}{\lambda\sqrt{f}} H_y 
\\
\frac{i}{\sqrt{f}} \bar H_y &   \frac{i}{\lambda\sqrt{f}} \bar H_y & \frac{1}{2} (B + \bar B)
\end{pmatrix},
\end{align*}
where $\lambda$ is defined by
\begin{align}\label{lambdadef}
\lambda^2 = \frac{k-(1-y)}{k-x},
\end{align}
and
\begin{align*}
A(x,y) &= \frac{1}{2f(x,y)} \left( \E_x(x,y) - 2\overline{H(x,y)} H_x(x,y) \right),
\\  
B(x,y) &= \frac{1}{2f(x,y)} \left( \E_y(x,y) - 2\overline{H(x,y)} H_y(x,y) \right), 
\\
f(x,y) &= \frac{\E(x,y) + \overline{\E(x,y)}}{2} - H(x,y) \overline{H(x,y)} = \Re \E (x,y) - |H(x,y)|^2,
\end{align*}
for two complex-valued functions $\mathcal{E},H \in C^2(\Int D)$. Here $D$ is defined by \eqref{Ddef}.

The appearance of the multivalued function $\lambda$ suggests considering the Riemann surface $\Surface_{(x,y)}$, $(x,y)\in D$, consisting of all pairs $(\lambda,k) \in \hat{\C}^2$ satisfying \eqref{lambdadef}, where $\hat{\C} =\C \cup \{ \infty \}$ denotes the Riemann sphere. Note that $\Surface_{(x,y)}$ contains the points $\infty^- =(-1,\infty)$, $\infty^+ =(1,\infty)$, and $x=(\infty,x)$. The surface $\Surface_{(x,y)}$ can be realized by introducing a branch cut from $x$ to $1-y$ in the complex $k$-plane and assigning to each $k \in \hat{\C} \setminus [x,1-y]$ two values $k^\pm=(\lambda(x,y,k^\pm),k) \in \Surface_{(x,y)}$, where
\[
\lambda(x,y,k^+) = \sqrt{\frac{k-(1-y)}{k-x}}
\]
is characterized by having positive real part and $\lambda(x,y,k^-) = -\lambda(x,y,k^+)$. The function
\[
F_{(x,y)}: \Surface_{(x,y)} \to \hat{\C},\; (\lambda,k) \mapsto z=\frac{1+\lambda}{1-\lambda}
\]
is a biholomorphism and thus $\Surface_{(x,y)}$ is topologically a sphere. The map $F_{(x,y)}$ maps points on the branch cut in $\Surface_{(x,y)}$ onto the unit circle in $\hat{\C}$ and the branch points $x$ and $1-y$ are mapped to $-1$ and $1$, respectively. The Riemann surface $\Surface_{(x,y)}$ as well as the map $F_{(x,y)}$ are visualized in Figure \ref{figSandF}.
\begin{figure}
	\begin{center}
		\begin{tikzpicture}	
		\draw[thick] (-6,0) -- (-1.5,0) -- (-0.5,1)-- (-5,1)--(-6,0);
		\draw[thick] (-6,-1.5) -- (-1.5,-1.5) -- (-0.5,-0.5)-- (-5,-0.5)--(-6,-1.5);
		\draw[inner color = black, outer color=black](-4,0.5) circle (0.05cm);
		\draw[inner color = black, outer color=black](-4,-1) circle (0.05cm);
		\draw[inner color = black, outer color=black](-2.5,0.5) circle (0.05cm);
		\draw[inner color = black, outer color=black](-2.5,-1) circle (0.05cm);
		\draw[dashed] (-3.95,0.5)--(-2.55,0.5);
		\draw[dashed] (-3.95,-1)--(-2.55,-1);
		\node at (-3.4,-2.5) {\small $\Surface_{(x,y)}$};
		\node at (-4,0.3) {\tiny $x$};
		\node at (-2.5,0.3) {\tiny $1-y$};
		\node at (-4,-1.2) {\tiny $x$};
		\node at (-2.5,-1.2) {\tiny $1-y$};
		
		\draw[thick,->] (0.5,0.5)--(1.5,0.5);
		\node at (1,0.8) {\small $F_{(x,y)}$};
		
		\draw[thick,->] (5,-2) -- (5,2);
		\draw[thick,->] (2.5,0) -- (7.5,0);
		\draw[dashed] (5,0) circle (1cm);
		\draw[inner color = black, outer color=black](4,0) circle (0.05cm);
		\draw[inner color = black, outer color=black](6,0) circle (0.05cm);
		\node at (5,2.3) {\small $\im z$};
		\node at (7.9,0) {\small $\re z$};
		\node at (3.7,-0.3) {\small $-1$};
		\node at (6.2,-0.3) {\small $1$};
		\node at (5,-2.5) {\small $\hat{\C}$};
		\end{tikzpicture}
		\caption{The two sheets of the Riemann surface $\Surface_{(x,y)}$ and the map $F_{(x,y)}$.} \label{figSandF}
	\end{center}
\end{figure}
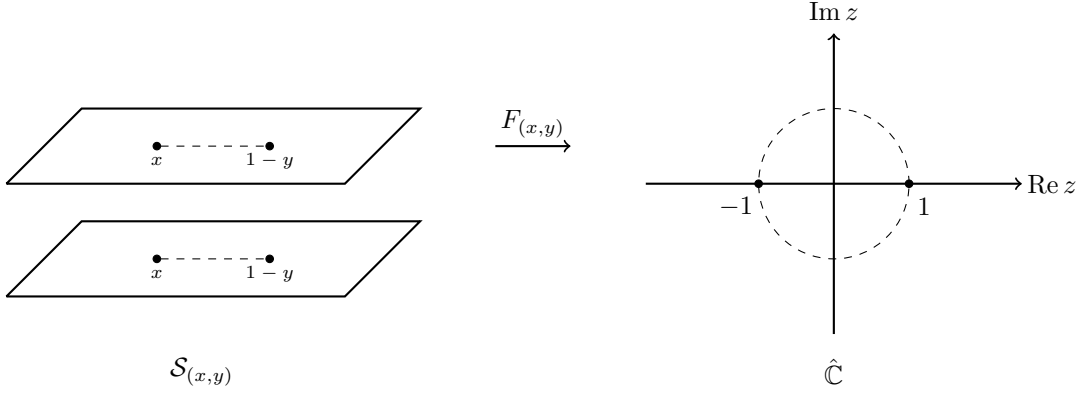

We define the oriented paths $\Sigma_{j} \subset \Surface_{(x,y)}$, $j=0,1$, by
\[
\Sigma_0= [0^+,x] \cup [x,0^-], \quad \Sigma_1 = [1^-,1-y] \cup [1-y,1^+],
\]
and the oriented simply closed curves $\Gamma_j \subset \C$, $j=0,1$, to be two clockwise oriented loops encircling $F_{(x,y)}(\Sigma_j)$, $j=0,1$, respectively, at some positive distance, but neither encircling nor intersecting zero. The contour $\Gamma$ is defined to be the union $\Gamma=\Gamma_0\cup \Gamma_1$ (see Figure \ref{figGamma}). By $\Omega_\infty$, $\Omega_0$, and $\Omega_1$ we denote the components of $\C \setminus \Gamma$ containing $0$, $-1$, and $1$, respectively (see Figure \ref{figOmega}).
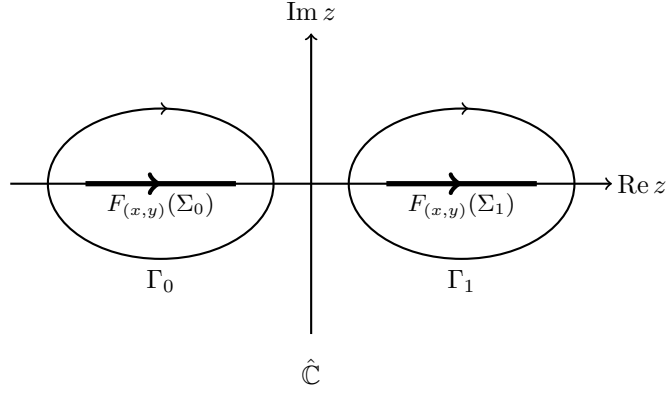
\begin{figure}
	\begin{center}
		\begin{tikzpicture}	
		
		\draw[thick,->] (5,-2) -- (5,2);
		\draw[thick,->] (1,0) -- (9,0);
		\node at (5,2.3) {\small $\im z$};
		\node at (9.4,0) {\small $\re z$};
		\node at (5,-2.5) {\small $\hat{\C}$};
		
		\draw[line width=0.7mm,->] (2,0)--(3,0);
		\draw[line width=0.7mm] (2.9,0)--(4,0);
		
		\draw[line width=0.7mm,->] (6,0)--(7,0);
		\draw[line width=0.7mm] (6.9,0)--(8,0);
		\node at (7,-0.3) {\footnotesize $F_{(x,y)}(\Sigma_1)$};
		\node at (3,-0.3) {\footnotesize $F_{(x,y)}(\Sigma_0)$};
		
		\draw[thick,decoration={markings, mark=at position 0.25 with {\arrow{<}}},
		postaction={decorate}] (7,0) ellipse (1.5cm and 1cm);
		\draw[thick,decoration={markings, mark=at position 0.25 with {\arrow{<}}},
		postaction={decorate}] (3,0) ellipse (1.5cm and 1cm);
		\node at (7,-1.3) {\small $\Gamma_1$};
		\node at (3,-1.3) {\small $\Gamma_0$};
	
		\end{tikzpicture}
		\caption{The contour $\Gamma=\Gamma_0\cup \Gamma_1$.} \label{figGamma}
	\end{center}
\end{figure}

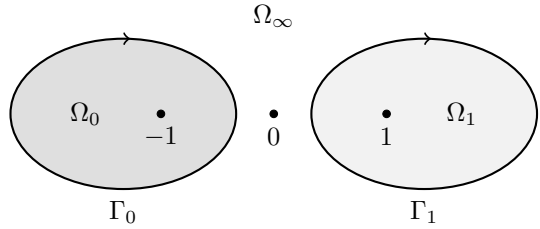
\begin{figure}
	\begin{center}
		\begin{tikzpicture}	
			
		\draw[inner color = black, outer color=black](5,0) circle (0.05cm);	
		\node at (5,-0.3) {\small $0$};

		\draw[thick,decoration={markings, mark=at position 0.25 with {\arrow{<}}},
		postaction={decorate}, inner color =gray!10, outer color=gray!10] (7,0) ellipse (1.5cm and 1cm);
		\draw[thick,decoration={markings, mark=at position 0.25 with {\arrow{<}}},
		postaction={decorate}, inner color =gray!25, outer color=gray!25] (3,0) ellipse (1.5cm and 1cm);
		\node at (7,-1.3) {\small $\Gamma_1$};
		\node at (3,-1.3) {\small $\Gamma_0$};
		\node at (7.5,0) {\small $\Omega_1$};
		\node at (2.5,0) {\small $\Omega_0$};
		\node at (5,1.3) {\small $\Omega_\infty$};
		
		\draw[inner color = black, outer color=black](3.5,0) circle (0.05cm);	
		\node at (3.5,-0.3) {\small $-1$};
		\draw[inner color = black, outer color=black](6.5,0) circle (0.05cm);	
		\node at (6.5,-0.3) {\small $1$};
		
		\end{tikzpicture}
		\caption{The regions $\Omega_\infty$, $\Omega_0$, and $\Omega_1$.} \label{figOmega}
	\end{center}
\end{figure}

Henceforth we consider $\lambda(x,y,P)$, $\Phi(x,y,P)$, $\U(x,y,P)$, and $\V(x,y,P)$,  $P \in \Surface_{(x,y)}$, as functions defined  on the Riemann surface $\Surface_{(x,y)}$. Assuming the functions $\mathcal{E}_0(x)$, $H_0(x)$, $\E_1(y)$, $H_1(y)$ satisfy the conditions \eqref{AssumptionsBoundary} for some $n \ge 2$, we define for $k \in \hat{\C} \setminus [0,1]$ and $(x,y) \in D$ the functions $\Phi_0$ and $\Phi_1$ to be the unique solutions of the Volterra equations
\begin{subequations}\label{phi0phi1def}
\begin{align}
	\Phi_0(x,k^\pm) &= I + \int_{0}^{x} \U_0(x',k^\pm )\Phi_0(x',k^\pm)dx',\label{phi0phi1defa}
	\\
	\Phi_1(y,k^\pm) &= I + \int_0^{y} \V_1(y',k^\pm) \Phi_1(x',k^\pm)dy',\label{phi0phi1defb}
\end{align}
\end{subequations}
where $\U_0$ and $\V_1$ are defined by
\begin{subequations} \label{DefUV}
	\begin{align}
	\U_0(x,k^\pm)&= \begin{pmatrix}
	\overline{A_0(x)} & \lambda(x,0,k^\pm) \overline{A_0(x)} & \frac{i}{\sqrt{f_0(x)}} H_{0x}(x)
	\\
	\lambda(x,0,k^\pm)  A_0(x) & A_0(x) & - \frac{i\lambda(x,0,k^\pm)}{\sqrt{f_0(x)}} H_{0x}(x) 
	\\
	\frac{i}{\sqrt{f_0(x)}} \overline{H_{0x}(x)} & \lambda(x,0,k^\pm)  \frac{i}{\sqrt{f_0(x)}} \overline{H_{0x}(x)} & \frac{1}{2} (A_0(x) + \overline{A_0(x)})
	\end{pmatrix}, 
	\\
	\V_1(y,k^\pm)&= \begin{pmatrix}
	\overline{B_1(y)} & \frac{1}{\lambda(0,y,k^\pm)} \overline{B_1(y)} & \frac{i}{\sqrt{f_1(y)}} H_{1y}(y))
	\\
	\frac{1}{\lambda(0,y,k^\pm)}  B_1(y) & B_1(y) & - \frac{i}{\lambda(0,y,k^\pm)\sqrt{f_1(y)}} H_{1y}(y) 
	\\
	\frac{i}{\sqrt{f_1(y)}} \overline{H_{1y}(y)} &   \frac{i}{\lambda(0,y,k^\pm)\sqrt{f_1(y)}} \overline{H_{1y}(y)} & \frac{1}{2} (B_1(y) + \overline{B_1(y)})
	\end{pmatrix},
	\end{align}
\end{subequations}
and
\begin{align*}
A_0(x) &= \frac{1}{2f_0(x)} \left( \E_{0x}(x) - 2\overline{H_0(x)} H_{0x}(x) \right),
\\
f_0(x) &= \frac{\E_0(x) + \overline{\E_0(x)}}{2} - H_0(x) \overline{H_0(x)}>0, 
\\
B_1(y) &= \frac{1}{2f_1(y)} \left( \E_{1y}(y) - 2\overline{H_1(y)} H_{1y}(y) \right), 
\\
f_1(y) &= \frac{\E_1(y) + \overline{\E_1(y)}}{2} - H_1(y) \overline{H_1(y)}>0.
\end{align*}
Note that it follows from \eqref{AssumptionsBoundary} that $f_0 >0$ and $f_1>0$. The following lemma takes care of existence and uniqueness of the eigenfunctions.
\begin{lemma} \label{lemma: uniqueness and existence phi0phi1}
For each fixed $k \in \hat{\C} \setminus [0,1]$, there exist unique solutions 
\begin{align*}
x\mapsto\Phi_0(x,k^\pm) \in C([0,1))\cap C^n((0,1)) \quad \mbox{and} \quad y\mapsto\Phi_1(y,k^\pm)\in  C([0,1))\cap C^n((0,1))
\end{align*}
of the Volterra equations \eqref{phi0phi1defa} and \eqref{phi0phi1defb}, respectively.
\end{lemma}
\begin{proof}
The proof is based on the method of successive approximations. We fix some compact subset $K\subset \C \setminus [0,1]$ and define
\begin{align} \label{defphi0second}
\Phi_0(x,k^\pm) = \sum_{j=0}^{\infty} \Phi_0^{(j)}(x,k^\pm), \qquad x \in [0,1), \quad k \in K,
\end{align}
where $\Phi_0^{(0)}=I$ and, for $j\ge 1$,
\begin{align} \label{defphi0j}
\Phi_0^{(j)}(x,k^\pm) = \int_{0 \le x_1 \le \cdots x_j\le x} \U_0(x_j,k^\pm) \U_0(x_{j-1},k^\pm)\cdots \U_0(x_1,k^\pm)dx_1 \cdots dx_j.
\end{align}
In view of \eqref{AssumptionsBoundary} and the fact that the functions $(x,k) \mapsto \lambda(x,0,k^\pm)$ are bounded on $[0,1) \times K$, we obtain the estimate
\begin{align*}
 \lVert \U_0(\cdot,k^\pm) \rVert_{L^1([0,x))} \le C(x),\qquad x\in [0,1),\quad k\in K,
\end{align*}
where the constant $C(x)$ is bounded on compact subsets of $[0,1)$. Consequently, we obtain
\begin{align} \label{estimatephi0j}
\big| \Phi_0^{(j)}(x,k^\pm)  \big| \le \frac{1}{j!}\lVert \U_0(\cdot,k^\pm) \rVert_{L^1([0,x))}^j \le \frac{C(x)^j}{j!}, \qquad x\in [0,1),\quad k\in K.
\end{align}
This implies that the series in \eqref{defphi0second} converges absolutely and uniformly for $k \in K$ and $x$ in compact subsets of $[0,1)$ to a solution of \eqref{phi0phi1defa}. The continuity of $\Phi_0$ follows from the fact that $x\mapsto \Phi_0^{(j)}(x,k^\pm)$ is continuous and the uniform and absolute convergence of the series in \eqref{defphi0second}. It is easy to see that $x\mapsto \Phi_0^{(j)}(x,k^\pm) \in C^n((0,1))$ and that the derivatives satisfy estimates analogous to \eqref{estimatephi0j}. Hence the uniform and absolute convergence in \eqref{defphi0second} and the differentiated versions of \eqref{defphi0second} shows that $x\mapsto \Phi_0(x,k^\pm) \in C^n((0,1))$.

Let $\tilde{\Phi}(x,k^\pm)$ be another continuous solution of \eqref{phi0phi1defa} and define $\Psi = \Phi_0-\tilde{\Phi}$. Then $\Psi$ satisfies the homogeneous equation
\begin{align*}
\Psi(x,k^\pm) = \int_{0}^{x} \U_0(x',k^\pm) \Psi(x',k^\pm) dx',\qquad x\in [0,1),\quad k\in K,
\end{align*}
which implies
\begin{align*}
\Psi(x,k^\pm)&= \int_{0 \le x_1 \le \cdots x_j\le x} \U_0(x_j,k^\pm) \U_0(x_{j-1},k^\pm)\cdots \U_0(x_1,k^\pm)\Psi(x_1,k^\pm) dx_1 \cdots dx_j,
\end{align*}
for all $x\in [0,1)$, $k\in K$, and $j\ge1$. Hence, for all $x\in [0,1)$ and $k \in K$,
\begin{align*}
\big| \Psi(x,k^\pm) \big| \le \sup_{x' \in [0,x]} \big| \Psi(x',k^\pm) \big| \frac{\lVert \U_0(\cdot,k^\pm) \rVert_{L^1([0,x))}^j }{j!} \to 0, \qquad j\to \infty.
\end{align*}
This yields $\Psi=0$ and hence $\Phi_0$ is the unique solution of \eqref{phi0phi1defa}. The proof for $\Phi_1$ is analogous.
\end{proof}
The following lemma summarizes some properties of the eigenfunctions $\Phi_0$ and $\Phi_1$.
\begin{lemma} \label{lemma: Propertiesphi0ph1}
	Let $\Phi_0$ and $\Phi_1$ be the unique solutions of the Volterra integral equations \eqref{phi0phi1def}. Then the following statements hold:
\begin{enumerate}[(i)]	
	\item For fixed $(x,y) \in D$, the functions $ k \mapsto \Phi_0(x,k^\pm)$ and $k \mapsto \Phi_1(y,k^\pm)$ are analytic on $\hat{\C} \setminus  [0,1]$. Moreover, $\Phi_0(x,\cdot )$ extends to an analytic function on $\Surface_{(x,y)} \setminus (\Sigma_0\cup \Sigma_1)$ and $\Phi_1(\cdot,y)$ extends to an analytic function on $\Surface_{(x,y)} \setminus (\Sigma_0\cup \Sigma_1)$ (for any $(x,y) \in D$).
	\item At the point $P=\infty^+$, we have for each $(x,y) \in D$
	\begin{subequations}\label{Phi0Phi1atinfinity}
		\begin{align} 
		\Phi_0(x,\infty^+) &= \frac{1}{2}\begin{pmatrix}
		\overline{\E_0(x)} - 2H_0(x) \overline{H_0(x)} &1& i H_0(x) \\
		\E_0(x) & -1&-iH_0(x) \\
		2i\overline{H_0(x)}\sqrt{f_0(x)}  &0&\sqrt{f_0(x)}
		\end{pmatrix}
		\begin{pmatrix}
		1 & 1 & 0 \\
		1 & -1 & 0 \\
		0 & 0 & 2 
		\end{pmatrix}, \label{Phi0Phi1atinfinitya}\\
		\Phi_1(y,\infty^+) &= \frac{1}{2}\begin{pmatrix}
		\overline{\E_1(y)} - 2H_1(y) \overline{H_1(y)} &1& i H_1(y) \\
		\E_1(y) & -1&-iH_1(y) \\
		2i\overline{H_1(y)}\sqrt{f_1(y)}  &0&\sqrt{f_1(y)}
		\end{pmatrix}
		\begin{pmatrix}
		1 & 1 & 0 \\
		1 & -1 & 0 \\
		0 & 0 & 2 
		\end{pmatrix}\label{Phi0Phi1atinfinityb}.
		\end{align}
	\end{subequations}
	
\item For all $(x,y) \in D $ and $P \in \Surface_{(x,y)} \setminus  (\Sigma_0\cup \Sigma_1)$, it holds that
	\begin{subequations}\label{detPhi0Phi1}
		\begin{align} 
		\det \Phi_0(x,P) &= f_0(x)^{\frac{3}{2}}, 
		\\
		\det \Phi_1(y,P) &= f_1(y)^{\frac{3}{2}}.
		\end{align}
	\end{subequations}
\item The functions $\Phi_0$ and $\Phi_1$ obey the symmetries  \begin{subequations}\label{SymmetriesPhi0Phi1}
\begin{align} 
	\Phi_0(x,k^+) &= \Lambda  \Phi_0(x,k^-) \Lambda ,& \overline{\Phi_0(x,\bar{k}^\pm)} &=  f_0(x) \Lambda  (\Phi_0(x,k^\pm)^{-1})^t \Lambda,  \label{SymmetriesPhi0Phi1a}
	\\
	\Phi_1(y,k^+) &= \Lambda  \Phi_1(y,k^-) \Lambda , &\overline{\Phi_1(y,\bar{k}^\pm)} &=  f_1(y) \Lambda  (\Phi_1(y,k^\pm)^{-1})^t \Lambda, \label{SymmetriesPhi0Phi1b}
	\end{align}
\end{subequations}
for $(x,y)\in D$, $k \in \hat{\C} \setminus [0,1]$,  where
\[
\Lambda = \begin{pmatrix}
1 & 0 & 0 
\\
0 & -1 & 0
\\
0 & 0 & 1
\end{pmatrix}.
\]
\item Let $(x_0,y_0) \in \Int D$ be fixed and $K \subset \hat{\C} \setminus ([0,x_0] \cup [1-y_0,1])$ be compact. Then the maps
\begin{subequations}\label{ktophi}
\begin{align}
x &\mapsto (k \mapsto \Phi_0(x,k^\pm)),\label{ktophia}
\\
y &\mapsto (k \mapsto \Phi_1(y,k^\pm))\label{ktophib}
\end{align}
\end{subequations}
are continuous as maps $[0,x_0) \to L^\infty(K)$ and $[0,y_0) \to L^\infty(K)$, and $C^n$ as maps $(0,x_0) \to L^\infty(K)$ and $(0,y_0) \to L^\infty(K)$, respectively. Furthermore, the maps
\begin{align*}
x &\mapsto (k \mapsto x^\alpha\Phi_0(x,k^\pm)),
\\
y &\mapsto (k \mapsto y^\alpha\Phi_1(y,k^\pm))
\end{align*}
are continuous as maps $[0,x_0) \to L^\infty(K)$ and $[0,y_0) \to L^\infty(K)$, respectively.
\end{enumerate}
\end{lemma}

\begin{proof}
We prove the statements (i)--(v) for $\Phi_0$. The proofs for $\Phi_1$ are similar. It follows from the definition of $\U_0$, equation \eqref{defphi0j}, and the uniform and absolute convergence of the series in \eqref{defphi0second} that $k\mapsto \Phi_0(x,k^\pm)$ is analytic on $\hat{\C} \setminus [0,1]$. Since the function $k\mapsto \U_0(x,k^\pm)$ has continuous boundary values onto $(x,1-y)$, the same is true for $k\mapsto \Phi_0(x,k^\pm)$ by extending the argument in the proof of Lemma \ref{lemma: uniqueness and existence phi0phi1} to compact sets $K$ reaching up to the branch cut $(x,1-y)$. Furthermore, since the values of $\U_0(x,k^+)$ and $\U_0(x,k^-)$ fit together across the branch cut $(x,1-y)$ of $\mathcal{S}_{(x,y)}$, the same is true for $\Phi_0(x,k^+)$ and $\Phi_0(x,k^-)$ by uniqueness of the solution of the Volterra equation \eqref{phi0phi1defa}. This completes the proof of (i).
	
Noting that $\lambda(x,0,\infty^+)=1$, we obtain that
\begin{align*}
\begin{pmatrix}
\overline{\E_0(x)} - 2H_0(x) \overline{H_0(x)}  \\
\E_0(x) \\
2i\overline{H_0(x)}\sqrt{f_0(x)} 
\end{pmatrix}, \quad
\begin{pmatrix}
1 \\
 -1 \\
0
\end{pmatrix}, \quad \mbox{and} \quad 
\begin{pmatrix}
 i H_0(x) \\
-iH_0(x) \\
\sqrt{f_0(x)}
\end{pmatrix}
\end{align*}
are linearly independent solutions of the equation
\begin{align*}
\Phi_{0x}(x,\infty^+)=\U_0(x,\infty^+)\Phi_{0}(x,\infty^+), \quad x\in (0,1).
\end{align*}
Using that $\mathcal{E}_0(0)=1$, $H_0(0)=0$, and $\Phi_0(0,\infty^+)=I$ yields \eqref{Phi0Phi1atinfinitya}, which proves (ii).

In order to prove (iii), we recall that every differentiable matrix-valued function $B(x)$ with values in $GL(n,\C)$ satisfies the relation
\begin{align*}
\ln(\det B(x))_x = \tr (B^{-1}B_x(x)).
\end{align*}
This implies
\begin{align*}
\ln(\det \Phi_0(x,P))_x = \tr\big(\Phi_0^{-1}(x,P) \U_0(x,P) \Phi_0(x,P)\big) =\tr\big(   \U_0(x,P)\big)= \frac{3}{2}( \ln f_0(x))_x
\end{align*}
for $P\in \Surface_{(x,y)} \setminus (\Sigma_0\cup \Sigma_1)$ and $x$ in a small neighborhood of the origin (since $\Phi_0(0,P)=I$). Hence there exists a constant $c(P)>0$ possibly depending on $P$ such that
\begin{align} \label{detPhi}
\det \Phi_0(x,P) = c(P) f_0(x)^{\frac{3}{2}}
\end{align}
for small $x$. A continuity argument and the positivity of $f_0(x)$ for all $x\in [0,1)$ imply that \eqref{detPhi} can be extended to all $x\in [0,1)$. The fact that $\det \Phi_0(0,P) =1=f_0(0)$ for all $P$ implies that $c(P)=1$ which completes the proof of (iii).

It follows from the relation $\lambda(x,0,k^+)=-\lambda(x,0,k^-)$ that
\begin{align*}
 \U_0 (x,k^+)=\Lambda \U_0(x,k^-) \Lambda.
\end{align*}
Hence uniqueness of the solution of \eqref{phi0phi1defa} yields the first symmetry in \eqref{SymmetriesPhi0Phi1a}. For the second symmetry, we observe that $\overline{\lambda(x,0,\bar{k}^\pm)} = \lambda(x,0,k^\pm)$. This yields
\[
\overline{\U_0(x,\bar{k}^\pm)} = (A_0(x))+\bar{A}_0(x)) I - \Lambda \U_0(x,k^\pm)^t \Lambda =  \frac{f_{0x}(x)}{f_0(x)} I - \Lambda \U_0(x,k^\pm)^t \Lambda.
\]
The matrix-valued function $f_0(x) \Lambda (\Phi_0(x,k^\pm)^{-1})^t \Lambda$ (note that $\Phi_0(x,k^\pm)$ is invertible by \eqref{detPhi0Phi1}) satisfies the same Volterra equation as $ \overline{\Phi_0(x,\bar{k}^\pm)} $, which is
\begin{align*}
\overline{\Phi_0(x,\bar{k}^\pm)} &= I + \int_0^x  \overline{\U_0(x',\bar{k}^\pm)} \, \overline{\Phi_0(x',\bar{k}^\pm)}dx',
\end{align*}
since, by a straightforward calculation, we have
\begin{align*}
	&I + \int_0^x  \overline{\U_0(x',\bar{k}^\pm)} \,f_0(x') \Lambda (\Phi_0(x',k^\pm)^{-1})^t \Lambda dx'
	\\
	=\; &  I + \int_0^x  \left(\frac{f_{0x}(x')}{f_0(x')} I - \Lambda \U_0(x',k^\pm)^t \Lambda \right)\,f_0(x') \Lambda (\Phi_0(x',k^\pm)^{-1})^t \Lambda dx'
	\\
	=\; &  I + [f_0(x') \Lambda (\Phi_0(x',k^\pm)^{-1})^t \Lambda]_{x'=0}^x
	\\
	& \quad+  \int_0^x  f_{0}(x') \Lambda (\Phi_0(x',k^\pm)^{-1} \U_0(x',k^\pm)\Phi_0(x',k^\pm) \Phi_0(x',k^\pm)^{-1})^t \Lambda  dx'
	\\
	& \quad -\int_0^x f_0(x')\Lambda \U_0(x',k^\pm)^t(\Phi_0(x',k^\pm)^{-1})^t \Lambda dx'
=f_0(x) \Lambda (\Phi_0(x,k^\pm)^{-1})^t \Lambda.
	\end{align*}
	Now the second symmetry in \eqref{SymmetriesPhi0Phi1a} follows by uniqueness of the solutions of the Volterra equation \eqref{phi0phi1defa}. This completes the proof of (iv).
	
	For the last part of the lemma, we observe the estimate
\begin{align*}
& \sup_{k \in K} \big|\Phi_0(x_2, k^+) - \Phi_0(x_1, k^+)\big|
= \sup_{k \in K} \bigg|\int_{x_1}^{x_2} (\mathsf{U}_0\Phi_0)(x, k^+) dx\bigg|
\\
& \leq \bigg( \sup_{k \in K} \sup_{x \in [0, x_0)} |x^\alpha \mathsf{U}_0(x,k^+)|\bigg) \sup_{k \in K} \left(\int_{x_1}^{x_2} |x^{-\alpha}\Phi_0(x, k^+)| dx \right) \\
&\leq C \sup_{k \in K} \left(\int_{x_1}^{x_2} |x^{-\alpha}\Phi_0(x, k^+)| dx \right), \qquad x_1, x_2 \in [0, x_0),
\end{align*}
for some constant $C>0$ and the right-hand side of the above equation tends to zero as $x_2 \to x_1$ by \eqref{estimatephi0j}. This shows that \eqref{ktophia} is continuous on $[0,x_0)$. Moreover, it holds that
\begin{align*}
\sup_{k \in K} \bigg| &\frac{\Phi_0(x+h, k^+) - \Phi_0(x,k^+)}{h} - \Phi_{0x}(x, k^+)\bigg|
\\
& \leq \sup_{k \in K} \bigg| \frac{1}{h} \int_x^{x+h} (\mathsf{U}_0\Phi_0)(x', k^+) dx' - \Phi_{0x}(x,k^+)\bigg|
\\
&  \leq \sup_{k \in K} \bigg| (\mathsf{U}_0\Phi_0)(\xi, k^+)  - (\mathsf{U}_0\Phi_0)(x, k^+)\bigg|
\end{align*}
for some $\xi \in (x,x+h)$ and the last expression in the above equation tends to zero as $h\to 0$. Hence the map \eqref{ktophia} is $C^1$ on $(0,x_0)$ and the fact that its derivative satisfies the relation $\Phi_{0x} = \U_0 \Phi_0$ proves all remaining statements. This completes the proof of part (v) and the lemma.
\end{proof}

\subsection{The Riemann--Hilbert problem}
Let $\mathcal{E}_0(x)$, $H_0(x)$, $\E_1(y)$, $H_1(y)$ satisfy the conditions \eqref{AssumptionsBoundary} for some $n \ge 2$. Defining the jump matrix $v(x,y,z)$ on $\Gamma = \Gamma_0 \cup \Gamma_1$ by
\begin{align} \label{jumpdef}
v(x,y,z) = \begin{cases}
\Phi_0\big(x,F_{(x,y)}^{-1}(z)\big), \qquad z \in \Gamma_0, \\
\Phi_1\big(y,F_{(x,y)}^{-1}(z)\big), \qquad z \in \Gamma_1,
\end{cases}
\qquad (x,y) \in D,
\end{align}
we will analyze the Goursat problem for the equations \eqref{ErnstMaxwellEquations} with boundary data 
$\mathcal{E}_0(x)$, $H_0(x)$, $\E_1(y)$, $H_1(y)$ by using the following family of classical RH problems parametrized by $(x,y) \in D$. We seek a $\C^{3\times3}$-valued function $m(x,y,z)$ such that
\begin{align}\label{RHm}
\begin{cases} 
\text{$m(x,y,\cdot)$ is analytic on $\C \setminus\Gamma$},
\\
\text{$m(x, y, \cdot)$ has continuous boundary values $m_+$ and $m_-$ on $\Gamma$},
\\
\text{$m_+(x, y, z) = m_-(x, y, z) v(x, y, z)$ for all $z \in \Gamma$},
\\
\text{$m(x,y,z) = I + O(z^{-1})$ as $z \to \infty$}.
\end{cases} 
\end{align}
Here the functions $m_+(x,y,z)$ and $m_-(x,y,z)$ denote the continuous boundary values of $m(x,y,z)$ as $z$ approaches the contour $\Gamma$ from the left- and right-hand side of $\Gamma$, respectively, according to its orientation. 

Since the jump matrix $v(x,y,z)$ has positive determinant for all $z \in \Gamma$ by \eqref{detPhi0Phi1}, uniqueness of the solution of the RH problem \eqref{RHm} can be proved by standard arguments (cf. e.g. \cite{D1999}) which also imply that $m(x,y,z)$ is invertible for all $z \in \C \setminus \Gamma$, and, in particular,
\begin{align} \label{determinantm}
	\det m(x,y,z) = 1, \quad (x,y) \in D, \quad z \in \Omega_\infty.
\end{align}

\section{Main results} \label{mainresultssec}
Throughout this section, we assume that $\mathcal{E}_0(x)$, $H_0(x)$, $\E_1(y)$, $H_1(y)$ are complex-valued functions satisfying the conditions \eqref{AssumptionsBoundary} for some $n \ge 2$. We define a \emph{$C^n$-solution of the Goursat problem for \eqref{ErnstMaxwellEquations} in $D$ with data $\{ \E_0,H_0, \E_1, H_1 \}$} to be a pair of complex-valued functions $\{ \mathcal{E},H \}$ satisfying
\begin{align*}
\begin{cases}
\E \in C(D) \cap C^n(\Int(D)), 
\\
\text{$\E(x,y)$ and $H(x,y)$ satisfy \eqref{ErnstMaxwellEquations} in $\Int(D)$,}
\\
\text{$x^\alpha \E_x,x^\alpha H_x, y^\alpha \E_y,y^\alpha H_y, x^\alpha y^\alpha \E_{xy}, x^\alpha y^\alpha H_{xy} \in C(D)$ for some $\alpha \in [0,1)$,}
\\
\text{$\E(x,0) = \E_0(x)$ and $H(x,0)=H_0(x)$ for $x \in [0,1)$,}
\\
\text{$\E(0,y) = \E_1(y)$ and $H(0,y)=H_1(y)$ for $y \in [0,1)$,}
\\
\text{$\re \E(x,y) - |H(x,y)|^2 >0$ for $(x,y) \in D$.}
\end{cases}
\end{align*}

Using the notation introduced in Section \ref{laxpairRHsection}, we will now present the three main results of the paper. The following theorem treats uniqueness of a $C^n$-solution of the Goursat problem for \eqref{ErnstMaxwellEquations} and its representation in terms of the RH problem \eqref{RHm}.

\begin{theorem}\label{mainth1}
	The $C^n$-solution $\{ \mathcal{E}, H \}$ of the Goursat problem for \eqref{ErnstMaxwellEquations} with data 
	\\$\{ \E_0,H_0, \E_1, H_1 \}$ is unique, if it exists, and satisfies
	\begin{subequations}\label{EHrecover}
		\begin{align}
		\E(x,y) &= \frac{(m(x,y,0))_{33}+(m(x,y,0))_{11}-(m(x,y,0))_{21} }{(m(x,y,0))_{33}+(m(x,y,0))_{11}+(m(x,y,0))_{21}}, \\ 
		H(x,y) &=\frac{-i(m(x,y,0))_{23}}{(m(x,y,0))_{33}+(m(x,y,0))_{11}+(m(x,y,0))_{21}},
		\end{align}
	\end{subequations}
		where $m(x,y,z)$ is the unique solution of the RH problem \eqref{RHm}. Furthermore, the values $\mathcal{E}(x,y)$ and $H(x,y)$ for some $(x,y) \in D$ depend only on the values $\mathcal{E}_0(x')$, $H_0(x')$, $\E_1(y')$, $H_1(y')$ for $x'\in [0,x]$ and $y'\in [0,y]$.

\end{theorem}
 
 The following theorem treats existence of a $C^n$-solution.

\begin{theorem} \label{mainth2}
	Whenever the RH problem \eqref{RHm} has a solution, then there exists a $C^n$-solution of the Goursat problem for \eqref{ErnstMaxwellEquations} in $D$ with data $\{ \E_0,H_0, \E_1, H_1 \}$. Furthermore, for some fixed $\delta >0$, there always exists a $C^n$-solution of the Goursat problem for \eqref{ErnstMaxwellEquations} in the smaller triangle $D_\delta = D \cap \{ (x,y):x+y < 1-\delta \}$ (see Figure \ref{figDdelta}) with data 
	\[ 
	\{ \E_0|_{[0,1-\delta)},H_0|_{[0,1-\delta)}, \E_1|_{[0,1-\delta)}, H_1|_{[0,1-\delta)} \}
	\]
	whenever the $L^1([0,1-\delta))$-norms of $A_0$, $B_1$, $H_0/\sqrt{f_0}$, and $H_1/\sqrt{f_1}$ are sufficiently small.
\end{theorem}
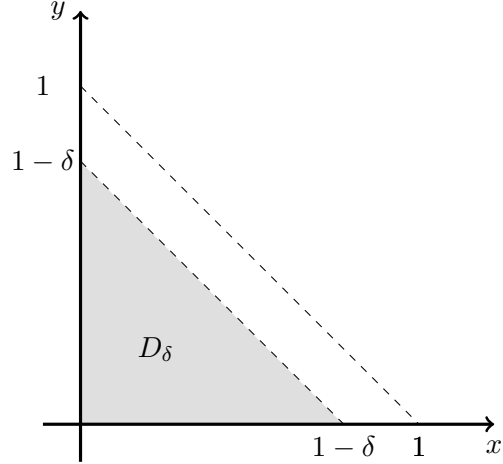
\begin{figure}
	\begin{center}
		\begin{tikzpicture}	
		\draw[white,shade, inner color = gray!25, outer color = gray!25] (-1.5,0) -- (-1.5,3.5) -- (2,0) -- (-1.5,0) ;
		\draw[very thick,->] (-2,0) -- (4,0);
		\draw[very thick,->] (-1.5,-0.5) -- (-1.5,5.5);
		\node at (4,-0.3) {$x$};
		\node at (-1.8,5.5) {$y$};
		
		\draw[dashed] (-1.5,4.5) -- (3,0);
		\draw[dashed] (-1.5,3.5) -- (2,0);
		\node at (-2,4.5) {$1$};
		\node at (3,-0.3) {$1$};
			\node at (-2,3.5) {$1-\delta$};
		\node at (2,-0.3) {$1-\delta$};
		\node at (-0.5,1) {$D_\delta$};

		\node at (3,-0.3) {$1$};		
		\end{tikzpicture}
		\caption{The slightly smaller triangle $D_\delta$.} \label{figDdelta}
	\end{center}
\end{figure}
 The following theorem treats the boundary values of a $C^n$-solution.
 
\begin{theorem}\label{mainth3}
	Assume that $\alpha \in (0,1)$ and that $\{ \mathcal{E},H \}$ is a $C^n$-solution of the Goursat problem for \eqref{ErnstMaxwellEquations} in $D$ with data $\{ \E_0,H_0, \E_1, H_1 \}$ and let
	\begin{align*}
		m_1&= \lim_{x \downarrow 0} x^\alpha \mathcal{E}_{0x}(x), & n_1&= \lim_{x \downarrow 0} x^\alpha H_{0x}(x),
		\\
		m_2&= \lim_{y \downarrow 0} y^\alpha \mathcal{E}_{1y}(y), & n_2&= \lim_{y \downarrow 0} y^\alpha H_{1y}(y).
	\end{align*}
	Then we have
	\begin{subequations}\label{boundaryconditionsEH}
	\begin{align}
			\lim_{x \downarrow 0} x^\alpha \mathcal{E}_x(x,y) &= c_1\frac{(im_1\bar{d_1} +2\bar{e}_1n_1)\bar{H_1}+c_1(e_1m_1+2id_1n_1) }{\sqrt{f_1}\sqrt{1-y}}, \label{boundaryconditionsEH1}
			\\
			\lim_{y\downarrow 0} y^\alpha \mathcal{E}_y(x,y) &= c_2\frac{(im_2\bar{d_2} +2\bar{e}_2n_2)\bar{H_0}+c_2(e_2m_2+2id_2n_2) }{\sqrt{f_0}\sqrt{1-x}},\label{boundaryconditionsEH2}
			\\
		\lim_{x \downarrow 0} x^\alpha H_x(x,y)	&= c_1\frac{(im_1\bar{d}_1 +2\bar{e}_1n_1) }{2\sqrt{f_1}\sqrt{1-y}},\label{boundaryconditionsEH3}
		\\
		\lim_{y \downarrow 0} y^\alpha H_y(x,y)	&= c_2\frac{(im_2\bar{d}_2 +2\bar{e}_2n_2) }{2\sqrt{f_0}\sqrt{1-x}},\label{boundaryconditionsEH4}
	\end{align}
\end{subequations}
where, if $(x,y) \in D$,
\begin{subequations}
\begin{align}
c_1(y)&=(\Phi_1(y, 0))_{22}=e^{\int_0^y B_1(y')dy'},  &d_1(y) &= (\Phi_1(y, 0))_{31}, & e_1(y) &= (\Phi_1(y, 0))_{33},
\\
c_2(x)&=(\Phi_0(x, 0))_{22}=e^{\int_0^x A_0(x')dx'}, \hspace{-0.1cm} &d_2(x) &= (\Phi_0(x, 0))_{31}, & e_2(x) &=  (\Phi_0(x, 0))_{33}. 
\end{align}
\end{subequations}
Moreover, if $\alpha=1/2$, it holds that
\begin{subequations}\label{boundaryrelevantforwaves}
\begin{align}
&(1-y)	\lim_{x\to 0} x \left( \frac{|\mathcal{E}_x(x,y)-2\bar{H}_1(y)H_x(x,y)|^2}{f_1(y)^2} + \frac{4 |H_x(x,y)|^2}{f_1(y)} \right)=|m_1|^2+4|n_1|^2,\label{boundaryrelevantforwavesa}
\\
&(1-x)	\lim_{y\to 0} y  \left( \frac{|\mathcal{E}_y(x,y)-2\bar{H}_0(x)H_y(x,y)|^2}{f_0(x)^2} + \frac{4 |H_y(x,y)|^2}{f_0(x)} \right)=|m_2|^2+4|n_2|^2.\label{boundaryrelevantforwavesb}
\end{align}
\end{subequations}
\end{theorem}

\section{Proofs of the main results} \label{proofssec}
In this section, we prove the three main results Theorems \ref{mainth1}-\ref{mainth3}. We note that the proofs are conceptually similar to those in \cite{LMernst}. In particular, some parts of the proofs rather rely on the analyticity structure of the matrices $\U$ and $\V$ than on their specific expression. These parts require almost no modification compared to \cite{LMernst}. However, for the purpose of completeness and the reader's convenience, we include all details of the proof. 

Throughout this section, we assume that $\E_0$, $H_0$,  $\E_1$, $H_1$ are complex-valued functions satisfying the assumptions \eqref{AssumptionsBoundary} for some $n\ge2$.

\subsection{Proof of Theorem \ref{mainth1}} Let $\{ \mathcal{E}, H \}$ be a $C^n$-solution of the Goursat problem for \eqref{ErnstMaxwellEquations} with data $\{ \E_0,H_0, \E_1, H_1 \}$. For $(x,y) \in D$, we define the eigenfunction $\Phi(x,y,k^\pm)$ of the Lax pair \eqref{LaxPair} to be the solution of the Volterra equation
\begin{align}\label{firstVolterraPhi}
\Phi(x,y,k^\pm) = \Phi_0(x,k^\pm) + \int_{0}^{y} \V(x,y',k^\pm) \Phi(x,y',k^\pm)dy', \quad k \in \hat{\C} \setminus [0,1].
\end{align}
The same arguments as in the proof of Lemma \ref{lemma: uniqueness and existence phi0phi1} show that $\Phi(x,y,k^\pm)$ exists and is the unique solution of \eqref{firstVolterraPhi} such that the function $(x,y) \mapsto \Phi(x,y,k^\pm)$ is in $C(D)\cap C^n(\Int D)$. Moreover, in the same way as in the proof of Lemma \ref{lemma: Propertiesphi0ph1}, the eigenfunction $P \mapsto \Phi(x,y,P)$ extends to an analytic function on $\Surface_{(x,y)} \setminus (\Sigma_0 \cup \Sigma_1)$. The following Lemma is a key ingredient of the proof of Theorem \ref{mainth1}.
\begin{lemma} \label{lemma: phiinversephianalytic}
Let $\Phi$ be the unique solution of the Volterra equation \eqref{firstVolterraPhi}. Then the functions
\[
P \mapsto \Phi(x,y,P)\Phi(x,0,P)^{-1} \text{ and } P \mapsto \Phi(x,y,P)\Phi(0,y,P)^{-1}
\]
extend to analytic functions on $\Surface_{(x,y)} \setminus \Sigma_1$ and $\Surface_{(x,y)} \setminus \Sigma_0$, respectively.
\end{lemma}
{\it Proof of Lemma \ref{lemma: phiinversephianalytic}}.
Let $P \in \Surface_{(x,y)} \setminus (\Sigma_0 \cup \Sigma_1)$. First we note that $\Phi(x,0,P)= \Phi_0(x,P)$ and hence $\Phi(x,0,P)$ is defined for $P \in \Surface_{(x,y)} \setminus (\Sigma_0 \cup \Sigma_1)$ by Lemma \ref{lemma: proeprties of m} (i) and invertible by \eqref{detPhi0Phi1}. We multiply \eqref{firstVolterraPhi} by $\Phi(x,0,P)^{-1}$ from the right which yields the Volterra equation
\begin{align} \label{Voterraforphiphiinverse}
\Phi(x,y,P)\Phi(x,0,P)^{-1}  = I + \int_{0}^{y} \V(x,y',P) \Phi(x,y',P)\Phi(x,0,P)^{-1}dy'
\end{align}
for $(x,y) \in D$. Since the function $P\mapsto \lambda^{-1}(x,y',P)$ is analytic on $\Surface_{(x,y)} \setminus \Sigma_1$, so is the function $P\mapsto \V(x,y',P)$. Consequently, the solution of \eqref{Voterraforphiphiinverse}, given by
\[
P \mapsto \Phi(x,y,P)\Phi(x,0,P)^{-1},
\]  
is analytic on $\Surface_{(x,y)} \setminus \Sigma_1$.

In order to prove the analogous statement for the function
\[
P \mapsto \Phi(x,y,P)\Phi(0,y,P)^{-1},
\]
it is clearly enough to show that the solution of \eqref{firstVolterraPhi} also satisfies the Volterra equation
\begin{align}\label{secondVolterraPhi}
\Phi(x,y,k^\pm) = \Phi_1(y,k^\pm) + \int_{0}^{x} \U(x',y,k^\pm) \Phi(x',y,k^\pm)dx', 
\end{align}
where $(x,y)\in D$ and $k \in \hat{\C}\setminus [0,1]$. In order to show this we first prove that $\Phi(x,y,k^\pm)$ satisfies the Lax pair equations \eqref{LaxPair}. It holds by \eqref{firstVolterraPhi} that $\Phi_y = \V \Phi$. The fact that $\{ \mathcal{E},H \}$ is a solution of the Goursat problem for \eqref{ErnstMaxwellEquations} implies that 
\begin{align}\label{UVcompactible}
\V_x = \U_y + \U \V - \V \U,
\end{align}
which is the compatibility condition of the Lax pair \eqref{LaxPair}. Furthermore, we have
\begin{align}\label{phi0satisfiesLaxpair}
\Phi_x(x,0,k^\pm) = \U(x,0,k^\pm) \Phi(x,0,k^\pm)
\end{align}
by \eqref{phi0phi1defa}. Differentiating \eqref{firstVolterraPhi} with respect to $x$ and substituting in \eqref{UVcompactible} and \eqref{phi0satisfiesLaxpair} yields
\begin{align*}
\Phi_x(x,y,k^\pm)&= \Phi_x(x,0,k^\pm) + \int_0^y (\V_x \Phi)(x,y',k^\pm) + (\V \Phi_x)(x,y',k^\pm) dy'
\\
&=  \U(x,0,k^\pm) \Phi(x,0,k^\pm) +  \int_{0}^{y} (\U_y \Phi)(x,y',k^\pm) + (\U \V \Phi) (x,y',k^\pm) dy'
\\
&\quad+ \int_{0}^{y} (\V \Phi_x)(x,y',k^\pm) - (\V \U \Phi)(x,y',k^\pm) dy'
\\
 &=\U(x,y,k^\pm) \Phi(x,y,k^\pm) + \int_{0}^{y} (\V \Phi_x)(x,y',k^\pm) - (\V \U \Phi)(x,y',k^\pm) dy',
\end{align*}
where we used
\begin{align*}
\U_y \Phi+ \U \V \Phi  = \U_y \Phi+ \U\Phi_y= (\U \Phi)_y.
\end{align*}
Consequently, $\Phi_x-\U \Phi$ is the unique solution of the homogeneous Volterra equation
\begin{align*}
\tilde{\Phi}(x,y,k^\pm) = \int_{0}^{y} (\V \Phi)(x,y',k^\pm)dy',
\end{align*}
which implies $\Phi_x = \U \Phi$. This proves that $\Phi$ satisfies the Lax pair equations \eqref{LaxPair}. Now the difference between \eqref{firstVolterraPhi} and \eqref{secondVolterraPhi} is given by
\begin{align*}
&\Phi_0(x,k^+) - \Phi_1(y,k^+) +\int_0^y \V \Phi(x,y',k^+)dy' - \int_0^x \U \Phi(x',y,k^+)dx' 
\\
=& \, \int_0^y \int_0^x (\V \Phi)_x(x',y',k^+)dx' dy' - \int_0^x \int_0^y (\U \Phi)_y(x',y',k^+)dy'dx' 
\\
=& \, \int_0^x \int_0^y (\V \Phi)_x(x',y',k^+)-(\U \Phi)_y(x',y',k^+) dy' dx'.
\end{align*}
But $(\V\Phi)_x = (\U \Phi)_y$ is precisely the compatibility condition for the Lax pair \eqref{LaxPair}. Hence we showed that $\Phi$ satisfies \eqref{secondVolterraPhi} and the same argument as for $\Phi(x,y,P)\Phi(x,0,P)^{-1} $ can be applied to show that $P \mapsto \Phi(x,y,P)\Phi(0,y,P)^{-1} $ extends to an analytic function on $\Surface_{(x,y)} \setminus \Sigma_0$. \proofendcontinue

Lemma \ref{lemma: phiinversephianalytic} shows that the function $m(x,y,z)$ defined by
\begin{align} \label{formulamuniqueness}
m(x,y,z)=\Phi(x,y,\infty^+)^{-1} \Phi(x,y,F_{(x,y)}^{-1}(z)) \times
\begin{cases}
\Phi(x,0,F_{(x,y)}^{-1}(z))^{-1}, &z \in \Omega_{0}, \\
\Phi(0,y,F_{(x,y)}^{-1}(z))^{-1}, &z \in \Omega_{1}, \\
I, &z \in \Omega_\infty,
\end{cases}
\end{align}
solves the RH problem \eqref{RHm}.

It is similar to the proof of Lemma \ref{lemma: Propertiesphi0ph1} part (ii) and (iv) to show that the eigenfunction $\Phi$ enjoys the symmetries
\begin{align}\label{phisymmetries}
\Phi(x,y,k^+) = \Lambda \Phi(x, y, k^-)\Lambda , \qquad \overline{\Phi(x,y,\bar{k}^\pm)} =  f(x,y) \Lambda  (\Phi(x,y,k^\pm)^{-1})^t \Lambda,
\end{align}
for $(x,y) \in D$ and $k \in \hat{\C} \setminus [0,1]$, and
\begin{align}\label{phiatinfinity}
	\Phi(x,y,\infty^+) = \frac{1}{2}\begin{pmatrix}
	\overline{\E(x,y)} - 2H(x,y) \overline{H(x,y)} &1& i H(x,y) \\
	\E(x,y) & -1&-iH(x,y) \\
	2i \overline{H(x,y)} \sqrt{f(x,y)}&0&\sqrt{f(x,y)}
	\end{pmatrix}
	\begin{pmatrix}
	1 & 1 & 0 \\
	1 & -1 & 0 \\
	0 & 0 & 2 
	\end{pmatrix},
\end{align}
for $(x,y) \in D$.
Hence we find
\begin{align}\label{Formula m(x,y,0)}
m(x,y,0) &= \Phi(x,y,\infty^+)^{-1} \Phi(x,y,\infty^-) \nonumber
\\
&= \Phi(x,y,\infty^+)^{-1} \Lambda  \Phi(x,y,\infty^+) \Lambda.
\end{align}
Substituting \eqref{phiatinfinity} into \eqref{Formula m(x,y,0)}, we can write the (11), (21), (33), and (23) entries of the right-hand side of (\ref{Formula m(x,y,0)}) as
\begin{align*}
(m(x,y,0))_{11} &=\frac{1+|\E(x,y)|^2- 2|H(x,y)|^2}{2(\re \E(x,y)) - 2|H(x,y)|^2},
\\
(m(x,y,0))_{21} &=\frac{(1-\E(x,y))(1+\overline{\E(x,y)})}{2(\re \E(x,y)) - 2|H(x,y)|^2},
\\
(m(x,y,0))_{33} &= \frac{(\re \E(x,y)) + |H(x,y)|^2 }{(\re \E(x,y)) - |H(x,y)|^2},
\\
(m(x,y,0))_{23} &= \frac{i(1+\overline{\E(x,y)})H(x,y) }{(\re \E(x,y)) - |H(x,y)|^2}.
\end{align*}
Solving these equations for $\E$, $\bar \E$, $H$, and $\bar H$ gives \eqref{EHrecover}. Thus we proved that the $C^n$-solution $\{ \mathcal{E},H \}$ can be represented by \eqref{EHrecover}. Since the solution of the RH problem \eqref{RHm}, whose formulation only involves the values of the boundary data on $[0,x]$ and $[0,y]$, is unique, this also completes the proof of Theorem \ref{mainth1}. \proofend

\subsection{Proof of Theorem \ref{mainth2}} Letting $w(x,y,z)=v(x,y,z)-I$, the existence of a solution of the Riemann--Hilbert problem \eqref{RHm} is equivalent to invertibility of the Cauchy operator (cf. \cite{L2018})
\[
I-\mathcal{C}_{w(x,y,\cdot)},
\]
which is a bounded linear operator $L^2(\Gamma)\to L^2(\Gamma)$. Here the operator $C_{{w(x,y,\cdot)}}$ is defined as follows: For some function $g \in L^2(\Gamma)$, we define the Cauchy integral $\mathcal{C}g$ by
\[
\mathcal{C}g(z)= \frac{1}{2\pi i} \int_{\Gamma} \frac{g(z')}{z'-z}dz', \quad z \in \C \setminus \Gamma,
\]
and by $\mathcal{C}_+g$ and $\mathcal{C}_-g$ we define the non-tangential limits of $\mathcal{C}g$ from the left- and right-hand side of $\Gamma$, respectively. Then the operator $\mathcal{C}_{w}$ is defined by
\[
\mathcal{C}_{w}(g)=\mathcal{C}_-(gw).
\]
Assuming the operator $(I-\mathcal{C}_{w(x,y,\cdot)})$ is invertible, it follows from the theory of singular integral equations that the unique solution $m$ of \eqref{RHm} can be written as (cf. \cite{D1999,L2018})
\begin{align} \label{mviacauchyop}
m = I + \mathcal{C}(\mu w),
\end{align}
where
\[
\mu(x,y,z) =I+ (I-\mathcal{C}_{w(x,y,\cdot)})^{-1}\mathcal{C}_{w(x,y,\cdot)}(I)(z).
\]

In order to study the solution $m(x,y,z)$ as a function of $x$ and $y$, it is useful to fix one contour $\Gamma$ for the family of RH problems \eqref{RHm} and to write the solution in the form \eqref{mviacauchyop} with this fixed $\Gamma$. However, as $x+y \to 1$, the intervals $F_{(x,y)}(\Sigma_0)$ and $F_{(x,y)}(\Sigma_1)$ come arbitrarily close to the origin and to infinity and hence it is impossible to choose one satisfying contour $\Gamma$ for all $(x,y) \in D$. We circumvent this problem by fixing some $\delta >0$ and considering the solution $m(x,y,z)$ of \eqref{RHm} for all $(x,y) \in D_\delta$, where $D_\delta$ is defined as in the statement of Theorem \ref{mainth2}. Then we can fix a contour $\Gamma$ having the following properties:
\begin{itemize}
	\item $\Gamma$ has positive distance to the segments $F_{(x,y)} (\Sigma_0)$ and $F_{(x,y)} (\Sigma_1)$ for all $(x,y) \in D_\delta$.
	\item $\Gamma$ is invariant under the involutions $z \to z^{-1}$ and $z \to \bar{z}$.
	\item $\Gamma$ satisfies all properties required in its definition in Section \ref{laxpairRHsection}.
\end{itemize}

Now the proof of Theorem \ref{mainth2} consists of two steps. First we assume that $(I-\mathcal{C}_{w(x,y,\cdot)})$ is invertible for all $(x,y) \in D$, define $m$ by \eqref{mviacauchyop} and $\{ \mathcal{E},H \}$ by \eqref{EHrecover}, and show all desired properties of the functions $\mathcal{E}$ and $H$ restricted to $D_\delta$. Since $\delta>0$ can be chosen arbitrarily small, this completes the proof of the first part of the theorem. The second part of the theorem treats invertibility of the operator $(I-\mathcal{C}_{w(x,y,\cdot)})$ for all $(x,y) \in D_\delta$ under a small-norm assumption on the boundary data.

For $(x,y) \in D_\delta$ and $z \in \C \setminus \Gamma$, we define $m$ to be the unique solution of the RH problem \eqref{RHm} given by
\[
m(x,y,z) = I + \frac{1}{2 \pi i} \int_{\Gamma} \frac{(\mu w)(x,y,z')}{z'-z}dz'.
\]
The following lemma treats the regularity and symmetry properties of $m$.
\begin{lemma} \label{lemma: proeprties of m}
The function $m$, defined in \eqref{mviacauchyop}, has the following properties:
	\begin{enumerate}[(i)]
	\item For all $z \in \C \setminus \Gamma$, the function $(x,y) \mapsto m(x,y,z)$ is continuous on $D_\delta$ and $C^n$ on $\Int D_\delta$. Furthermore, the functions $(x,y) \mapsto x^\alpha m_x(x,y,z)$, $(x,y) \mapsto y^\alpha m_y(x,y,z)$, and $(x,y)\mapsto x^\alpha y^\alpha m_{xy}(x,y,z)$ are continuous on $D_\delta$.
	\item For all $(x,y) \in D_\delta$, the function $m(x,y,z)$ obeys the symmetries
	\begin{align}\label{msymm}
	m(x,y,z) = m(x,y,0) \Lambda   m(x,y,z^{-1})\Lambda  , \qquad z \in \hat{\C} \setminus \Gamma,
	\end{align}
	and
	\begin{align}\label{msymm2}
	\overline{m(x,y,\bar{z})} = \begin{cases}
	\frac{1}{f_0(x)} \Lambda (m(x,y,z)^{-1})^t \Lambda, \quad &z \in \Omega_0, 
	\\
	\frac{1}{f_1(y)}\Lambda  (m(x,y,z)^{-1})^t \Lambda, \quad &z \in \Omega_1, 
	\\
	\Lambda(m(x,y,z)^{-1})^t \Lambda, \quad &z \in \Omega_\infty.
	\end{cases}
	\end{align}
\end{enumerate}
\end{lemma}
{\it Proof of Lemma \ref{lemma: proeprties of m}}. 
We first prove part (i). We will show that the maps
\begin{align}
D_\delta \to L^\infty(\Gamma),\quad&(x,y) \mapsto (z\mapsto w(x,y,z)),\label{mapsformregularitya}
\end{align}
and
\begin{align}
D_\delta \to L^2(\Gamma),\quad &(x,y) \mapsto (z\mapsto \mu(x,y,z)),\label{mapsformregularityb}
\end{align}
are continuous on $D_\delta$ and $C^n$ on $\Int D_\delta$. Let $(x_0,y_0) \in D_\delta$. We fix a small neighborhood $U\subset D_\delta $ of $(x_0,y_0)$ such that there exists a compact set $K\subset \hat{\C}$ with $ F_{(x,y)}^{-1}(\Gamma)\subset K^\pm$ and $K \cap([0,x) \cup (1-y,1])=\emptyset$ for all $(x,y) \in U$. Part (i) and (v) of Lemma \ref{lemma: Propertiesphi0ph1} show that the maps
\begin{align}\label{xymapstophi}
(x,y) \mapsto  (k\mapsto \Phi_0(x,k^\pm)), \qquad (x,y)\mapsto (k\mapsto \Phi_1(y,k^\pm))
\end{align}
are continuous as maps $U \to C^n(K)$ and $C^n$ as maps $\Int U \to C^n(K)$, where $C^n(K)$ denotes the space of $n$ times continuously differentiable functions on $K$ equipped with the norm
\[
\lVert g \rVert_{C^n} = \max_{|\alpha|\le n} \lVert D^\alpha g\rVert_{L^\infty}.
\] 
The map \eqref{mapsformregularitya} restricted to $U$ can be realized as composition of the maps \eqref{xymapstophi} and the map
\begin{align}
\mathcal{F}:U\to \mathcal{B}(C^n(K),C(\Gamma)), \quad (x,y) \mapsto \big(g \mapsto g\circ F_{(x,y)}^{-1}\big),
\end{align}
where $\mathcal{B}(\cdot,\cdot)$ the space of bounded linear operators between two linear normed spaces equipped with the operator norm. In the definition of $\mathcal{F}$, we always consider the $k$-projection of $(\lambda,k)=F_{(x,y)}^{-1}(z)$. We will show that the map $\mathcal{F}$ is continuous on $U$ and $C^n$ on $\Int U$. Continuity of $\mathcal{F}$ follows from
 \begin{align*}
 \sup_{\|g\|_{C^n}=1} \sup_{z \in \Gamma} \big|g(F_{(x,y)}^{-1}(z)) - g(F_{(x',y')}^{-1}(z))\big| \to 0
 \end{align*}
 as $(x',y') \to (x,y)\in U$ by uniform continuity of $g$. Differentiability of $\mathcal{F}$ on $\Int U$ follows from
 \begin{align*}
\sup_{\|g\|_{C^n}=1} \sup_{z \in \Gamma} \bigg|\frac{g(F_{(x+h,y)}^{-1}(z)) - g(F_{(x,y)}^{-1}(z))}{h} - \frac{d}{dx}g(F_{(x,y)}^{-1}(z))  \bigg| \to 0
 \end{align*}
  as $h \to  0$. Analogous limits for the higher derivatives show that $\mathcal{F}$ is $C^n$. Since $(x_0,y_0)$ was arbitrary, this proves that the map \eqref{mapsformregularitya} is continuous on $D_\delta$ and $C^n$ on $\Int D_\delta$. Now the fact that the maps
  \begin{align*}
  g\mapsto I-\mathcal{C}_g: L^\infty(\Gamma) \to \mathcal{B}(L^2(\Gamma)), \quad g\mapsto \mathcal{C}_-g: L^2(\Gamma) \to L^2(\Gamma)
  \end{align*}
  are smooth implies that \eqref{mapsformregularityb} is continuous on $D_\delta$ and $C^n$ on $\Int D_\delta$. Similar arguments show that the maps
  \begin{align*}
(x,y) &\mapsto (z\mapsto x^\alpha w_x(x,y,z)), & (x,y) &\mapsto (z\mapsto x^\alpha \mu_x(x,y,z)),
\\
(x,y) &\mapsto (z\mapsto y^\alpha w_y(x,y,z)), & (x,y) &\mapsto (z\mapsto y^\alpha \mu_y(x,y,z)),
\\
(x,y) &\mapsto (z\mapsto x^\alpha y^\alpha w_{xy}(x,y,z)), & (x,y) &\mapsto (z\mapsto x^\alpha  y^\alpha \mu_{xy}(x,y,z))
 \end{align*}
are continuous as maps $D_\delta \to L^\infty(\Gamma)$ and $D_\delta \to L^2(\Gamma)$, respectively. Now part (i) follows from \eqref{mviacauchyop}.
  
For part (ii), we note that the symmetries \eqref{SymmetriesPhi0Phi1} imply
\begin{align*}
v(x,y,z^{-1})=\Lambda v(x,y,z) \Lambda,\qquad\overline{ v(x,y,\bar{z})}= f_j(x) \Lambda \big(v(x,y,z)^{-1}\big)^t \Lambda,
\end{align*}
for $z \in \Gamma_j$, $j=0,1$. Hence part (ii) follows from uniqueness of the solution of the RH problem \eqref{RHm}.
\proofendcontinue
	
Let $(x,y) \in D_\delta$. We define $\hat{m}(x,y) = m(x,y,0)$. Then the symmetry \eqref{msymm} implies
\[
\hat{m}(x,y)^{-1} = \Lambda \hat{m}(x,y)  \Lambda.
\]
Together with equation \eqref{determinantm} this gives
\begin{align} \label{adjm}
\text{adj}(\hat{m})=
\begin{pmatrix}
\hat{m}_{11} & -\hat{m}_{12} & \hat{m}_{13}
\\ 
-\hat{m}_{21} & \hat{m}_{22} & -\hat{m}_{23}
\\
\hat{m}_{31} & -\hat{m}_{32} & \hat{m}_{33}
\end{pmatrix},
\end{align}
where $\text{adj}(\hat{m})$ denotes the adjugate matrix of $\hat{m}$. After some calculations, this yields
\begin{subequations}\label{symmetriesm0}
	\begin{align}
	\hat{m}_{33} &=1-\hat{m}_{11}+\hat{m}_{22}, \label{symmetriesm01}
	\\ 
	\hat{m}_{12}\hat{m}_{21} &= (\hat{m}_{11}-1)(\hat{m}_{22}+1), \label{symmetriesm02}
	\\
	\hat{m}_{13}\hat{m}_{31} &=(\hat{m}_{11}-1)(\hat{m}_{22}-\hat{m}_{11}),\label{symmetriesm03}
	\\
	\hat{m}_{23}\hat{m}_{32}&=(\hat{m}_{22}+1)(\hat{m}_{22}-\hat{m}_{11}).\label{symmetriesm04}
	\end{align}
\end{subequations}
Furthermore, the symmetries \eqref{msymm} and \eqref{msymm2} imply
\begin{align} \label{mtranspose}
	\overline{\hat{m}} = \hat{m}^t.
\end{align}
Hence the identities \eqref{symmetriesm02} and \eqref{symmetriesm04} yield
\[
|\hat{m}_{21}|^2 + |\hat{m}_{23}|^2=(\hat{m}_{22}+1)(\hat{m}_{11}-1 + \hat{m}_{22}-\hat{m}_{11}) = \hat{m}_{22}^2-1,
\]
which gives $(\hat{m}(x,y))_{22} \in (-\infty,-1] \cup [1,\infty)$. For $(x,y) = (0,0)$ we have $m(0,0,z) = I$ for all $z$, because the jump matrix $v$ becomes the identity matrix. In particular, we obtain $\hat{m}_{22}(0,0)=1$. By continuity, this gives $(\hat{m}(x,y))_{22} \geq 1$ for all  $(x,y) \in D_\delta$. This, together with the identities (\ref{symmetriesm0}), implies $(\hat{m}(x,y))_{11} \geq 1$ and $(\hat{m}(x,y))_{33} \geq 1$ for all $(x,y) \in D_\delta$.

A straightforward algebraic computation then shows that
\begin{align}\label{mhatexpression}
\hat{m}(x,y) = \tilde{\Phi}(x,y)^{-1} \Lambda \tilde{\Phi}(x,y)\Lambda, \qquad (x,y) \in D_\delta,
\end{align}
where the $3\times 3$-matrix valued function $\tilde{\Phi}(x,y)$ is defined by
\begin{align}\label{tildePhidef}
\tilde{\Phi}(x,y) = \frac{1}{2}\begin{pmatrix}
\overline{\E(x,y)} - 2H(x,y) \overline{H(x,y)} &1& i H(x,y)
\\
\E(x,y) & -1&-iH(x,y) 
\\
2i \overline{H(x,y)} \sqrt{f(x,y)}&0&\sqrt{f(x,y)}
\end{pmatrix}
\begin{pmatrix}
1 & 1 & 0 
\\
1 & -1 & 0 
\\
0 & 0 & 2 
\end{pmatrix},
\end{align}
and the functions $\mathcal{E}(x,y)$, $\overline{\mathcal{E}(x,y)}$, $H(x,y)$, and $ \overline{H(x,y)}$ are defined by
\begin{subequations}\label{EHdef}
	\begin{align}
	\mathcal{E} &= \frac{\hat{m}_{33}+ \hat{m}_{11} - \hat{m}_{21}}{\hat{m}_{33} + \hat{m}_{11} + \hat{m}_{21}}, \qquad
	\bar{\mathcal{E}} = -\frac{1 - \hat{m}_{11} + \hat{m}_{21}}{1 - \hat{m}_{11} - \hat{m}_{21}}, \label{EHdef1}
	\\
	H&=i\frac{-\hat{m}_{23}}{\hat{m}_{11}+\hat{m}_{21}+\hat{m}_{33}}, \quad \bar{H}= i\frac{\hat{m}_{21}(\hat{m}_{33}-1)}{(\hat{m}_{11}+\hat{m}_{21}-1)\hat{m}_{23}},\label{EHdef2}
	\end{align}
\end{subequations}
and $f=(\E + \bar{\E})/2 -H\bar{H}$ as before. Recalling the identities \eqref{symmetriesm0} and  \eqref{mtranspose} it follows that $\bar{\mathcal{E}}$ is the complex conjugate of $\mathcal{E}$ and $\bar{H}$ is the complex conjugate of $H$ whenever both expressions are well defined. Note that we have
\begin{align} \label{expressionf}
\nonumber f&= \frac{2 \hat{m}_{21}}{(\hat{m}_{33}+\hat{m}_{11}+\hat{m}_{21})(\hat{m}_{11}+\hat{m}_{21}-1)}
\\ 
&= \frac{2(\hat{m}_{11}-1)(\hat{m}_{22}+1)}{\hat{m}_{12}(1+\hat{m}_{22}+\hat{m}_{21})\big(\hat{m}_{11}+\frac{(\hat{m}_{11}-1)(\hat{m}_{22}+1)}{\hat{m}_{12}}-1\big)} 
\\
&= \frac{2(\hat{m}_{22}+1)}{(1+\hat{m}_{22}+\hat{m}_{21})(\hat{m}_{12}+\hat{m}_{22}+1)}= \frac{2(\hat{m}_{22}+1)}{|1+\hat{m}_{22}+\hat{m}_{21}|^2} >0,\nonumber
\end{align}
which proves $f >0$, whenever all terms in \eqref{EHdef} are well defined, i.e. all denominators are non-zero. Next we show that $\mathcal{E}$ and $H$ are free of singularities and have the desired regularity properties.

The first equation in \eqref{EHdef1} shows that $\mathcal{E}(x,y)$ has the same regularity properties as $\hat{m}(x,y)$ except possibly on the set
\begin{align}\label{singular1}
\{(x,y) \in D_\delta \, | \, (\hat{m}(x,y))_{11} + (\hat{m}(x,y))_{21} = -(\hat{m}(x,y))_{33}\}
\end{align}
where the denominator vanishes. 
In the same way, the second equation in \eqref{EHdef1} shows that $\mathcal{E}(x,y)$ is regular away from the set
\begin{align}\label{singular2}
\{(x,y) \in D_\delta \, | \, (\hat{m}(x,y))_{11} + (\hat{m}(x,y))_{21} = 1\}.
\end{align}
Since $\hat{m}_{33} \ge 1$, the sets \eqref{singular1} and \eqref{singular2} are disjoint and closed in $D_\delta$. Recalling the regularity properties of $m$, we conclude that $\mathcal{E} \in C(D_\delta) \cap C^n(\Int D_\delta)$. For $H$, the same argument is valid if $\hat{m}_{23} \neq 0$. Suppose that $(\hat{m}(x,y))_{23} = 0$ and $(\hat{m}(x,y))_{11} + (\hat{m}(x,y))_{21} = -(\hat{m}(x,y))_{33}$. Then $(\hat{m}(x,y))_{21}$ must be real, i.e. $(\hat{m}(x,y))_{21}=(\hat{m}(x,y))_{12}$. From $(\hat{m}(x,y))_{23} =0$ and \eqref{symmetriesm04} it follows that $(\hat{m}(x,y))_{11} =(\hat{m}(x,y))_{22}$ and thus $(\hat{m}(x,y))_{33} =1$ by \eqref{symmetriesm0}. Then \eqref{symmetriesm02} and $ (\hat{m}(x,y))_{21} = -(\hat{m}(x,y))_{33}-(\hat{m}(x,y))_{11}$ imply
\[
((\hat{m}(x,y))_{22} -1)((\hat{m}(x,y))_{22} +1)=(\hat{m}(x,y))_{21} ^2 = ((\hat{m}(x,y))_{22} +1)^2,
\]
which is a contradiction. We conclude that $H \in C(D_\delta) \cap C^n(\Int D_\delta)$. The same argument shows that $1+\hat{m}_{21}+ \hat{m}_{22} \neq 0$ and hence, by \eqref{expressionf}, $f(x,y)>0$ for all $(x,y) \in D_\delta$. In view of the regularity properties of $m$ and \eqref{expressionf}, we also proved $x^\alpha \E_x$, $x^\alpha H_x$, $y^\alpha \E_y$, $y^\alpha H_y$, $x^\alpha y^\alpha \E_{xy}$, $x^\alpha y^\alpha H_{xy} \in C(D_\delta)$. 

We show next that $\E(x,0) = \E_0(x)$ and $H(x,0)=H_0(x)$ for $x \in [0,1-\delta)$ and $\E(0,y) = \E_1(y)$ and $H(0,y)=H_1(y)$ for $y \in [0,1-\delta)$. From the properties of $\Phi_0$ given in Lemma \ref{lemma: Propertiesphi0ph1}, we see that the function
\begin{align} \label{def of m_0}
m_0(x,z)= \Phi_0(x,\infty^+)^{-1} \times \begin{cases}
I, \quad &z\in \Omega_0,
\\
\Phi_0\big( x,F_{(x,0)}^{-1}(z) \big), \quad &z \in \Omega_1 \cup \Omega_\infty,
\end{cases}
\end{align}
satisfies the RH problem \eqref{RHm} with $(x,y)=(x,0)$. After substituting in \eqref{Phi0Phi1atinfinitya}, a straightforward computation shows
\begin{align*}
(m_0(x,0))_{11} &=\frac{1+|\E_0(x)|^2- 2|H_0(x)|^2}{2(\re \E_0(x)) - 2|H_0(x)|^2},
&
(m_0(x,0))_{21} &=\frac{(1-\E_0(x))(1+\overline{\E_0(x)})}{2(\re \E_0(x)) - 2|H_0(x)|^2},
\\
(m_0(x,0))_{33} &= \frac{(\re \E_0(x)) + |H_0(x)|^2 }{(\re \E_0(x)) - |H_0(x)|^2},
&
(m_0(x,0))_{23} &= \frac{i(1+\overline{\E_0(x)})H_0(x) }{(\re \E_0(x)) - |H_0(x)|^2}.
\end{align*}
Solving these equations for $\E_0$, $\bar{\E_0}$, $H$, and $\bar{H}$ yields
\begin{align}
\E_0(x) &= \frac{(m_0(x,0))_{33}+(m_0(x,0))_{11}-(m_0(x,0))_{21} }{(m_0(x,0))_{33}+(m_0(x,0))_{11}+(m_0(x,0))_{21}}, \\ 
H_0(x) &=\frac{-i(m_0(x,0))_{23}}{(m_0(x,0))_{33}+(m_0(x,0))_{11}+(m_0(x,0))_{21}}.
\end{align}
But we have $m_0(x,z)=m(x,0,z)$ by uniqueness of the solution of the RH problem \eqref{RHm} and hence, recalling \eqref{EHdef}, we see that $\E_0(x)=\E(x,0)$ and $H_0(x)=H(x,0)$. The proof for $\E_1$ and $H_1$ is similar.

It only remains to show that $\{ \mathcal{E},H \}$ satisfies the hyperbolic Ernst--Maxwell equations \eqref{ErnstMaxwellEquations} in order to complete the proof of the  first part of Theorem \ref{mainth2}. Motivated by \eqref{formulamuniqueness}, we define for $(x,y) \in D_\delta$ and $P \in F_{(x,y)}^{-1}(\Omega_\infty)$
\begin{align}\label{def Phi in existence}
\Phi(x,y,P)=\tilde{\Phi}(x,y)m(x,y,F_{(x,y)}(P)).
\end{align}
Then the map $(x,y) \mapsto \Phi(x,y,P)$ is $C^n$ by the regularity properties of $m$.
\begin{lemma}\label{lemma: Phi satisfies lax pair}
We define $\U$ and $\V$ to be the matrices defined in \eqref{DefUV}, where $\E$ and $H$ are given by \eqref{EHdef}. Then the function $\Phi(x,y,P)$ defined in \eqref{def Phi in existence} satisfies the Lax pair equations \eqref{LaxPair}.
\end{lemma}
{\it Proof of Lemma \ref{lemma: Phi satisfies lax pair}}.  The idea of the proof is to use the analyticity structure of $\Phi_x$. Differentiating \eqref{def Phi in existence} yields for $P \in F_{(x,y)}^{-1}(\Omega_\infty)$
\begin{align} \label{Phi in existence differentiated}
\Phi_x(x,y,P)& = \tilde{\Phi}_x(x,y)m(x,y,z(x,y,P))
+ \tilde{\Phi}(x,y)\mathfrak{f}(x,y,z(x,y,P)) m(x,y,z(x,y,P)),
\end{align}
where $z(x,y,P)=F_{(x,y)}(P)$, and
\begin{align*}
\mathfrak{f}(x,y,z) = \big[m_x(x,y,z) + z_x\big(x,y, F_{(x,y)}^{-1}(z)\big) m_z(x,y,z)\big]m(x,y,z)^{-1}, \qquad z \in \hat{\C} \setminus \Gamma.
\end{align*}
Equation \eqref{mviacauchyop} and the identity
\begin{align*}
z_x\big(  x,y,F_{(x,y)}^{-1}(z)\big) = -\frac{1-z}{1+z} \frac{z}{1-x-y}
\end{align*}
imply that $\mathfrak{f}$ has a simple pole at $z=-1$ and is analytic at $z=\infty$ with $\mathfrak{f}(x,y,\infty)=0$. Hence $\mathfrak{f}$ is analytic on $\hat{\C} \setminus\big( \Gamma \cup \{ -1 \}\big)$ with the jump
\begin{align*}\nonumber
\mathfrak{f}_{+}(x,y,z) = &\; \mathfrak{f}_{-}(x,y,z) + m_-(x,y,z)\big[v_x(x,y,z) + z_x\big(x,y, F_{(x,y)}^{-1}(z)\big) v_z(x,y,z)\big]
\\ \label{f0jump}
& \times v(x,y,z)^{-1} m_-(x,y,z)^{-1}, \qquad z \in \Gamma,
\end{align*}
across $\Gamma$. Note that we used analyticity of $v$ in a small neighborhood of $\Gamma$, which follows from Lemma \ref{lemma: Propertiesphi0ph1} (i). By the definition of the jump matrix $v$, we have
\begin{align}
\begin{cases}
v_x(x,y,z) + z_x\big(x,y, F_{(x,y)}^{-1}(z)\big) v_z(x,y,z) = \Phi_{0x}(x,F_{(x,y)}^{-1}(z)),  \quad &z \in \Gamma_0,
\\
v_x(x,y,z) + z_x\big(x,y, F_{(x,y)}^{-1}(z)\big) v_z(x,y,z) = 0,  \quad &z \in \Gamma_1.
\end{cases}
\end{align}
Hence, the function
\begin{align}\label{ndef}
	\mathfrak{n}(x,y,z) = \begin{cases}
	\mathfrak{f}(x,y,z) + m(x,y,z)\mathsf{U}_0\big(x,F_{(x,y)}^{-1}(z)\big)m(x,y,z)^{-1}, \quad &  z \in \Omega_0,	\\
	\mathfrak{f}(x,y,z), & z \in \Omega_1 \cup \Omega_\infty,
	\end{cases}
\end{align}
is analytic on $\hat{\C} \setminus \{-1\}$ with a simple pole at $z=-1$ and $\mathfrak{n}(x,y,\infty)=0$. This implies that
\begin{align*}
		\mathfrak{f}(x,y,z)= \mathfrak{n}(x,y,z) = \frac{	\mathfrak{n}(x,y,0)}{z+1} =\frac{m_x(x,y,0)m(x,y,0)^{-1}}{z+1}
\end{align*}
for all $z \in \Omega_\infty$. Substituting this identity into \eqref{Phi in existence differentiated} yields
\begin{align*}
\Phi_x(x,y,P)= \bigg(\tilde{\Phi}_x(x,y)\tilde{\Phi}(x,y)^{-1} 
+ \tilde{\Phi}(x,y) \frac{\hat{m}_x(x,y)\hat{m}(x,y)^{-1}}{z(x,y,P)+1}\tilde{\Phi}(x,y)^{-1}\bigg)\Phi(x,y,P).
\end{align*}
Recalling \eqref{tildePhidef} and \eqref{EHdef}, a straightforward calculation yields $\Phi_x= \U \Phi$. The identity $\Phi_y = \V \Phi$ can be shown in similar way.
 \proofendcontinue

Fixing a point $P=(\lambda,k) \in F_{(x,y)}^{-1}(\Omega_\infty)$, we have
\[
\Phi_{xy}(x,y,P)-\Phi_{yx}(x,y,P)=0,
\]
since $\Phi(x,y,P)$ is $C^n$. In view of Lemma \ref{lemma: Phi satisfies lax pair}, this implies
\begin{align} \label{LaxPaireq}
\U_y-\V_x +[\U,\V]=0.
\end{align}
The $(23)$-entry of equation \eqref{LaxPaireq} reads
$$\frac{i(1-x-y)\lambda}{(f(x,y))^{3/2} (k-(1-y))}\bigg\{f\bigg(H_{xy} - \frac{H_x + H_y}{2(1-x-y)}\bigg) -\frac{1}{2}(\mathcal{E}_y H_x+\mathcal{E}_xH_y)+2\bar{H}H_xH_y \bigg\} = 0,$$
which implies that $\E$ and $H$ satisfy the second equation in \eqref{ErnstMaxwellEquations} for all $(x,y) \in D_\delta$.

The $(21)$-entry of equation \eqref{LaxPaireq} reads
\begin{align*}\frac{-(1-x-y)\lambda}{2(f(x,y))^2 (k-(1-y))}\bigg\{f\bigg(\mathcal{E}_{xy} - \frac{\mathcal{E}_x + \mathcal{E}_y}{2(1-x-y)}\bigg) - \mathcal{E}_x \mathcal{E}_y +\bar{H}(\E_xH_y+\E_yH_x) 
\\
-2\bar{H}\left[ f\bigg(H_{xy} - \frac{H_x + H_y}{2(1-x-y)}\bigg) -\frac{1}{2}(\mathcal{E}_y H_x+\mathcal{E}_xH_y)+2\bar{H}H_xH_y \right]\bigg\} = 0.
\end{align*}
Since $\E$ and $H$ satisfy the second equation in \eqref{ErnstMaxwellEquations} for all $(x,y) \in D_\delta$, it follows that the second term in the curly brackets vanishes. Consequently, $\E$ and $H$ satisfy the first equation in \eqref{ErnstMaxwellEquations} for all $(x,y) \in D_\delta$. 

This completes the proof of the first part of the theorem. The following lemma shows the second part.
\begin{lemma} \label{lemma: boundary data small}
There exists a constant $c_\delta>0$ such that the following statement holds: If
\begin{align*}
\lVert A_0 \rVert_{L^1([0,1-\delta)}, \lVert B_1 \rVert_{L^1([0,1-\delta)},\bigg\lVert \frac{H_0}{\sqrt{f_0}} \bigg\rVert_{L^1([0,1-\delta)},\bigg\lVert {\frac{H_1}{\sqrt{f_1}}} \bigg\rVert_{L^1([0,1-\delta)} < c_\delta,
\end{align*}
then the linear bounded operator $I-\mathcal{C}_{w(x,y,\cdot)}:L^2(\Gamma) \to L^2(\Gamma)$ is invertible.
\end{lemma}
{\it Proof of Lemma \ref{lemma: boundary data small}}. The idea is to invert $I-\mathcal{C}_{w(x,y,\cdot)}$ in terms of a Neumann series. It is clear from \eqref{defphi0second} and \eqref{estimatephi0j} and the analogs for $\Phi_1$ that
\begin{align*}
\lVert w(x,y,z) \rVert_{L^\infty(\Gamma)} < \lVert \mathcal{C}_- \rVert_{\mathcal{B}(L^2(\Gamma))}^{-1}
\end{align*}
for sufficiently small $c_\delta >0$, where $\lVert \cdot \rVert_{\mathcal{B}(L^2(\Gamma))}$ denotes the operator norm for bounded linear operators $L^2(\Gamma)\to L^2(\Gamma)$. This implies
\begin{align*}
\lVert \mathcal{C}_w \rVert_{\mathcal{B}(L^2(\Gamma))} \le \lVert \mathcal{C}_- \rVert_{\mathcal{B}(L^2(\Gamma))}\lVert w(x,y,z) \rVert_{L^\infty(\Gamma)} <1,
\end{align*}
and hence $I-\mathcal{C}_{w(x,y,\cdot)}$ is invertible. \proofendcontinue

We proved the following for arbitrary $\delta>0$: If the RH problem \eqref{RHm} has a solution for all $(x,y) \in D_\delta$, then there exists a solution $\{\E,H\}$, given by \eqref{EHdef}, of the Goursat problem for \eqref{ErnstMaxwellEquations} with boundary data $\{ \E_0,H_0,\E_1,H_1 \}$ on the slightly smaller triangle $D_\delta$. In particular, by Lemma \ref{lemma: boundary data small}, this is the case if the $L^1([0,1-\delta))$-norms of $A_0$, $B_1$, $H_0/\sqrt{f_0}$, and $H_1/\sqrt{f_1}$ are sufficiently small. Since $\delta>0$ was arbitrary, this completes the proof of Theorem \ref{mainth2}.
\proofend

\subsection{Proof of Theorem \ref{mainth3}}
Let $\mathcal{E}_0(x)$, $H_0(x)$, $x \in [0, 1)$, and $\mathcal{E}_1(y)$, $H_1(y)$, $y \in [0,1)$, be complex-valued functions satisfying \eqref{AssumptionsBoundary} for some $n \geq 2$. Suppose $\{ \mathcal{E}(x,y), H(x,y) \}$ is a $C^n$-solution of the Goursat problem for \eqref{ErnstMaxwellEquations} in $D$ with data $\{\mathcal{E}_0, H_0, \mathcal{E}_1,H_1\}$. We denote
$$m_1 = \lim_{x \downarrow 0}  x^\alpha \mathcal{E}_{0x}(x), \qquad
m_2 = \lim_{y \downarrow 0}  y^\alpha \mathcal{E}_{1y}(y), \quad n_1 =\lim_{x \downarrow 0}  x^\alpha H_{0x}(x), \quad n_2 = \lim_{y \downarrow 0}  y^\alpha H_{1y}(y),$$
and we want to compute the limits
\begin{align*}
& \lim_{x \downarrow 0} x^\alpha \mathcal{E}_x(x,y), \quad \lim_{y \downarrow 0} y^\alpha \mathcal{E}_y(x,y), \quad \lim_{x \downarrow 0} x^\alpha H_x(x,y), \quad \lim_{y \downarrow 0} y^\alpha H_y(x,y), \quad \text{$(x,y) \in D$},
\end{align*}
at the boundary of $D$. The idea is to use the representation of $m$ given in \eqref{mviacauchyop} together with the differentiated version of the representation formula \eqref{EHrecover}. We already computed the value of $m(x,y,0)$ in \eqref{Formula m(x,y,0)}. Hence it remains to compute $\lim_{x \downarrow 0} x^\alpha m_x(x,y,0)$ and $\lim_{y \downarrow 0} y^\alpha m_y(x,y,0)$. We will only consider the first limit since the second one is analogous.

It is clear from the definition of $w$ that
	\begin{align}\label{w(0)}
w(0,y, z) = \begin{cases}
0, \quad & z \in \Gamma_0, 
\\
\Phi_1\big(y, F_{(0,y)}^{-1}(z)\big) - I, & z \in \Gamma_1,
\end{cases} \quad y \in [0,1).
\end{align}
Noting that $m(x,y,0)=m_0(x,y)$ in \eqref{def of m_0}, the identity $\mu(x,y,z)= m_-(x,y,z)$ for $(x,y) \in D$ and $z \in\Gamma$ implies
\begin{align}\label{mu(0)}
\mu(0,y,z) = \begin{cases}
\Phi_1\big(y,\infty^+\big)^{-1}  \Phi_1\big(y,F_{(0,y)}^{-1}(z)\big), \quad & z \in \Gamma_0, 
\\
\Phi_1\big(y,\infty^+\big)^{-1} , & z \in \Gamma_1,
\end{cases}\quad y \in [0,1).
\end{align}
Furthermore, it follows from the definition of $w$ that $\lim_{x \downarrow 0} x^\alpha w_x(x,y,z)=0$ for $z \in \Gamma_1$ and
\begin{align*}
\lim_{x \downarrow 0} x^\alpha w_x(x,y,z)= \lim_{x \downarrow 0} x^\alpha \Phi_{0x}(x,F_{(x,y)}^{-1}(z)) = \lim_{x \downarrow 0} x^\alpha \U_0(x,F_{(x,y)}^{-1}(z))
\end{align*}
for $z \in \Gamma_0$. This yields
\begin{align}\label{xalphawx}
&\lim_{x \to 0} x^\alpha w_x(x,y, z) \nonumber \\
&= \begin{cases}
\frac{1}{2} 
\begin{pmatrix} \overline{m_1} & \overline{m_1} \lambda(0,0, F_{(0,y)}^{-1}(z)) & 2in_1 \\
m_1 \lambda(0,0, F_{(0,y)}^{-1}(z))  & m_1 & -2i\lambda(0,0, F_{(0,y)}^{-1}(z)) n_1 \\
2i\overline n_1 & 2i\lambda(0,0, F_{(0,y)}^{-1}(z)) \bar n_1 & \frac{1}{2} (m_1+\overline
m_1)
\end{pmatrix}, \; & z \in \Gamma_0, 
\\
0, & z \in \Gamma_1.
\end{cases}
\end{align}
Next we show that
\begin{align}\label{xalphamux}
&\lim_{x \downarrow 0} x^\alpha \mu_x(x,y,z) \nonumber
\\
&= 
- \frac{1}{\sqrt{1-y}}\frac{\Phi_1\big(y, \infty^+\big)^{-1}}{z+1}
\Phi_1\big(y, 0\big)\begin{pmatrix} 0 & \bar{m}_1 & 0 \\
m_1& 0 & -2in_1 \\ 0 & 2i \bar n_1 & 0 \end{pmatrix}\Phi_1\big(y, 0\big)^{-1}=: \Pi(y,z), \quad  z \in \Gamma_1.
\end{align}
Differentiating the identity $\mu = I + \mathcal{C}_w \mu$ gives $\mu_x=(I-\mathcal{C}_w)^{-1} \mathcal{C}_-(\mu w_x)$. Hence, in order to prove \eqref{xalphamux}, it suffices to show that $\mathcal{C}_-(\lim_{x \downarrow 0}x^\alpha \mu w_x)(z)=(I-\mathcal{C}_{w(0,y,\cdot)}) \Pi(z)$. Using \eqref{mu(0)} and \eqref{xalphawx}, we obtain for $z \in \Gamma_1$
\begin{align*}\nonumber
&\Big(\mathcal{C}_- \big[\lim_{x \to 0} x^\alpha  \mu(x,y,\cdot) w_x(x,y,\cdot)\big]\Big)(z)
= \frac{\Phi_1\big(y,\infty^+\big)^{-1} }{2\pi i} 
\\\nonumber
& \times \int_{\Gamma_0} \frac{\Phi_1\big(y,F_{(0,y)}^{-1}(z')\big)}{z' -z}
	\frac{1}{2} \Bigg(\begin{smallmatrix} \overline{m_1} & \overline{m_1} \lambda(0,0, F_{(0,y)}^{-1}(z)) & 2in_1 \\
	m_1 \lambda(0,0, F_{(0,y)}^{-1}(z))  & m_1 & -2i\lambda(0,0, F_{(0,y)}^{-1}(z)) n_1 \\
	2i\overline n_1 & 2i\lambda(0,0, F_{(0,y)}^{-1}(z)) \bar n_1 & \frac{1}{2} (m_1+\overline
	m_1)
	\end{smallmatrix}\Bigg) dz'
\\ \label{Cmuxalphawx}
& = 
- \Phi_1\big(y,\infty^+\big)^{-1} \underset{z' = -1}{\res} 
\frac{\Phi_1(y,F_{(0,y)}^{-1}(z'))  \lambda(0,0, F_{(0,y)}^{-1}(z'))\left(\begin{smallmatrix} 0 & \bar{m}_1 & 0 \\
	m_1& 0 & -2in_1 \\ 0 & 2i \bar n_1 & 0 \end{smallmatrix}\right)}{2(z' -z)} =: \tilde{\Pi}(y,z).
\end{align*}
We observe that $\Phi_1(y,F_{(0,y)}^{-1}(-1))=\Phi_1(y,0)$ and
\begin{align*}
\lambda(0,0, F_{(0,y)}^{-1}(z))= \sqrt{\frac{1}{(z+1)^2}\bigg(\frac{(z-1)^2 -y(z+1)^2}{1-y} \bigg)},
\end{align*}
where the square root has positive (negative) real part for $|z|>1$ ($|z|<1$). Hence we have
\begin{align*}
\lambda(0,0, F_{(0,y)}^{-1}(z))=\frac{1}{z+1} \sqrt{\frac{(z-1)^2 -y(z+1)^2}{1-y} },
\end{align*}
where the square root has negative values for $z<0$ (and in particular for $z=-1$). This yields
\begin{align*}
\tilde{\Pi}(y,z)=-\frac{\Phi_1\big(y,\infty^+\big)^{-1} \Phi_1(y,0)  \left(\begin{smallmatrix} 0 & \bar{m}_1 & 0 \\
	m_1& 0 & -2in_1 \\ 0 & 2i \bar n_1 & 0 \end{smallmatrix}\right) }{(z+1)\sqrt{1-y}}.
\end{align*}
Similarly, by deforming the contour $\Gamma_1$ to infinity, we obtain
\begin{align*}
& (\mathcal{C}_{w(0,y,\cdot)} \Pi)(z)
= \frac{1}{2\pi i} \int_{\Gamma_1}
\frac{\Pi(y, z')\big(\Phi_1\big(y, F_{(0,y)}^{-1}(z')\big) - I\big)}{z' - z_-} dz'
\\
& = -\frac{\Phi_1\big(y, \infty^+\big)^{-1}\Phi_1\big(y, 0\big)\left(\begin{smallmatrix} 0 & \bar{m}_1 & 0 \\
	m_1& 0 & -2in_1 \\ 0 & 2i \bar n_1 & 0 \end{smallmatrix}\right)\Phi_1\big(y, 0\big)^{-1}}{\sqrt{1-y}} 
\underset{z' = -1}{\res}\frac{\Phi_1\big(y, F_{(0,y)}^{-1}(z')\big) - I}{z' - z} \frac{1}{z'+1}
\\
&= \frac{\Phi_1\big(y, \infty^+\big)^{-1}\Phi_1\big(y, 0\big) \left(\begin{smallmatrix} 0 & \bar{m}_1 & 0 \\
	m_1& 0 & -2in_1 \\ 0 & 2i \bar n_1 & 0 \end{smallmatrix}\right)\Phi_1\big(y, 0\big)^{-1}}{\sqrt{1-y}} \frac{\Phi_1(y, 0) - I}{z+1}.
\end{align*}
Now a straightforward computation shows $\tilde{\Pi}(z)=(I-\mathcal{C}_{w(0,y,\cdot)}) \Pi(z)$ and hence \eqref{xalphamux} is proven.

After substituting the identities \eqref{w(0)}, \eqref{mu(0)}, \eqref{xalphawx}, and \eqref{xalphamux} into the equation
\[
m_x =  \mathcal{C}(\mu_x w) + \mathcal{C}(\mu w_x),
\] 
 similar residue calculations as the ones for proving \eqref{xalphamux} show
\begin{align}\nonumber
\lim_{x \to 0} x^\alpha m_x(x,y,0) 
= & - \frac{1}{\sqrt{1-y}}\Phi_1\big(y, \infty^+\big)^{-1}
\Phi_1\big(y, 0\big)\begin{pmatrix} 0 & \bar{m}_1 & 0 \\
m_1& 0 & -2in_1 \\ 0 & 2i \bar n_1 & 0 \end{pmatrix}
\\\label{formulaxalphamx}
&\times \Phi_1\big(y, 0\big)^{-1}\Lambda \Phi_1(y,\infty^+) \Lambda,
\end{align}
where we also used the symmetry \eqref{SymmetriesPhi0Phi1b}. It remains to compute the value of $\Phi_1(y,0)$.

\begin{lemma} \label{lemmaPhi10}For $y \in [0,1)$, it holds that
	\begin{align}\label{formulaPhi10}
	\Phi_1(y, 0) = \begin{pmatrix} a_1(y) & 0 & b_1(y) \\ 0 & c_1(y) & 0 \\ d_1(y) & 0 & e_1(y) \end{pmatrix},
	\end{align}
	for some functions $a_1,b_1,c_1,d_1,e_1$ satisfying 
\begin{subequations}\label{symmetriesPhi1(0,y)}
	\begin{align}
	e_1&=\overline{a_1}\sqrt{f_1}/c_1, &d_1&=-\overline{b_1}\sqrt{f_1}/c_1, &a_1e_1-b_1d_1&=f_1^{3/2    } /c_1,
	\\
	c_1&= f_1/\bar{c_1}, &	c_1(y) &= e^{\int_0^y B_1(y')dy'}.&&
	\end{align}
\end{subequations}
\end{lemma}
{\it Proof of Lemma \ref{lemmaPhi10}.}
This is a consequence of equation \eqref{detPhi0Phi1},  the symmetries \eqref{SymmetriesPhi0Phi1b}, and solving the equation 
\[
\Phi_{1y}(y,0)=\V_1(y,0) \Phi_1(y,0)
\]
 using the fact that $\lambda(0,y,0)=\infty$. \proofendcontinue
 
Now differentiating \eqref{EHrecover} and using \eqref{Phi0Phi1atinfinityb}, \eqref{Formula m(x,y,0)}, and \eqref{formulaxalphamx} gives the identities \eqref{boundaryconditionsEH1} and \eqref{boundaryconditionsEH3} under consideration of the symmetries \eqref{symmetriesPhi1(0,y)} (after a long but straightforward calculation). The other identities in \eqref{boundaryconditionsEH} can be shown in a similar way.

The identity \eqref{boundaryrelevantforwavesa} follows from \eqref{boundaryconditionsEH} and the symmetries \eqref{symmetriesPhi1(0,y)}. The proof of \eqref{boundaryrelevantforwavesb} is analogous. This completes the proof of Theorem \ref{mainth3}. \proofend

\section{Examples} \label{examplessec}
In this section, following \cite{Griffiths1991}, we will consider a simple class of solutions of the Goursat problem for \eqref{ErnstMaxwellEquations} and compute their boundary values \eqref{boundaryconditionsEH}.
Let
\begin{align} \label{tzcoordinates}
t=\sqrt{x}\sqrt{1-y} + \sqrt{y}\sqrt{1-x}, \qquad z=\sqrt{y}\sqrt{1-x} - \sqrt{x}\sqrt{1-y}.
\end{align}
Then a family of smooth solutions  of the Goursat problem for \eqref{ErnstMaxwellEquations} in $D$ with $\alpha=1/2$ is given by the potentials
\[
\mathcal{E} \equiv 1 , \qquad H=pt+iqz,
\]
where $p,q$ are real numbers such that $p^2+q^2=1$. For these solutions, we compute
\[
m_1=0=m_2, \qquad  n_1 = \frac{1}{2} (p-iq)=\bar{n}_2,
\]
and an easy computation yields
\begin{align*}
&(1-y)	\lim_{x\to 0} x \left( \frac{|\mathcal{E}_x(x,y)-2\bar{H}_1(y)H_x(x,y)|^2}{f_1(y)^2} + \frac{4 |H_x(x,y)|^2}{f_1(y)} \right)=1=|m_1|^2+4|n_1|^2,
\\
&(1-x)	\lim_{y\to 0} y  \left( \frac{|\mathcal{E}_y(x,y)-2\bar{H}_0(x)H_y(x,y)|^2}{f_0(x)^2} + \frac{4 |H_y(x,y)|^2}{f_0(x)} \right)=1=|m_2|^2+4|n_2|^2.
\end{align*}
Note that this family of solutions can be seen as a generalization of the famous Nutku--Halil solution \cite{NH1977} of the respective Goursat problem for the hyperbolic version of Ernst's equation \cite{E1968,Griffiths1991}, since the potentials
\[
\frac{1+H}{1-H}
\]
define the Nutku--Halil solutions of the hyperbolic version of Ernst's equation. The case of $p=1$ and $q=0$ is the Bell--Szekeres solution, which was first described in \cite{BS1974}.

\appendix
\renewcommand{\theequation}{A.\arabic{equation}}\nequation
\section{Colliding electromagnetic plane waves} \label{Aapp}
Let
\[
\mathcal{E}-|H|^2 = f+ i\phi,
\]
where $2\phi=i(\bar{\mathcal{E}}-\mathcal{E})$. We define a potential $\chi$ by
\[
\chi = (1-x-y)/f
\]
and a potential $\omega$ by the equations
\begin{align*}
	\omega_t&= \frac{1-z^2}{f^2}\big(\phi_z-i(H\bar{H}_z-\bar{H}H_z)\big), \\
	\omega_z&=\frac{1-t^2}{f^2}\big(\phi_t-i(H\bar{H}_t-\bar{H}H_t)\big),\\
\end{align*}
where we used the coordinate transformation \eqref{tzcoordinates}.

The Ernst--Maxwell equations can be written as \cite{Griffiths1990,Griffiths1991}
\begin{align} \label{ernstmaxwellbeforetrafo}
\begin{cases}
	2(\chi_{uv}+i\omega_{uv}) = U_u(\chi_v + i\omega_v)+U_v(\chi_u+ i\omega_u) 
	\\
	\quad +2\chi^{-1}(\chi_u+i\omega_u)(\chi_v+i\omega_v)+e^U 4\chi^2   H_v \bar{H}_u,
	\\
	0 =2\chi H_{uv}+(\chi_v +i\omega_v)H_u+(\chi_u-i\omega_u)H_v ,
	\end{cases}
\end{align}
where $e^{-U(u,v)}=1-x(v)-y(u)$ for some monotonically increasing functions $x(v)$ and $y(u)$. It is possible \cite{Szekeres1972} to use $x,y$ as coordinates which transforms \eqref{ernstmaxwellbeforetrafo} into \eqref{ErnstMaxwellEquations}. Here the second equation of \eqref{boundaryrelevantforwaves} directly becomes the second equation of \eqref{ErnstMaxwellEquations}. The first equation in \eqref{ernstmaxwellbeforetrafo} reduces to the first equation in \eqref{ErnstMaxwellEquations} after using the already observed second equation of \eqref{ErnstMaxwellEquations}.

In order to be relevant for colliding electromagnetic plane waves, a solution $\{ \mathcal{E},H \}$ of the Goursat problem for \eqref{ErnstMaxwellEquations} in $D$ must satisfy the boundary conditions (cf. Eq. (25) in \cite{Griffiths1990}; Note that there is a factor $1-x-y=f+g$ missing that comes from Eq. (22) and (23))
\begin{align*}
&(1-y)	\lim_{x\to 0} x \left( \frac{\chi_x^2+w_x^2}{\chi^2} + \frac{4 \chi}{1-x-y} H_x\bar{H}_x \right)=2k_1,
\\
&(1-x)	\lim_{y\to 0} y \left( \frac{\chi_y^2+w_y^2}{\chi^2} + \frac{4 \chi}{1-x-y} H_y\bar{H}_y \right)=2k_2,
\end{align*} 
for some constants $k_1,k_2\in [1/2,1)$. This is due to the fact that the resulting space time should be at least $C^1$. Recalling the definitions of $\chi$ and $\omega$, these boundary conditions become
\begin{align*}
&(1-y)	\lim_{x\to 0} x \left( \frac{|\mathcal{E}_x(x,y)-2\bar{H}_1(y)H_x(x,y)|^2}{f_1(y)^2} + \frac{4 |H_x(x,y)|^2}{f_1(y)} \right)=2k_1,
\\
&(1-x)	\lim_{y\to 0} y  \left( \frac{|\mathcal{E}_y(x,y)-2\bar{H}_0(x)H_y(x,y)|^2}{f_0(x)^2} + \frac{4 |H_y(x,y)|^2}{f_0(x)} \right)=2k_2.
\end{align*}
By \eqref{boundaryrelevantforwaves}, the limits above are indeed constant and only depend on the given boundary data. In particular, it is sufficient to check that
\[
| m_1|^2 + 4 |n_1|^2, | m_2|^2 + 4 |n_2|^2 \in [1,2).
\]

\bigskip
\noindent
{\bf Acknowledgement} {\it Support is acknowledged from the European Research Council, Grant Agreement No.~682537.}

\bibliographystyle{plain}

\end{document}